\documentclass[amssymb,twoside,11pt]{article}
\usepackage{latexsym,amsmath,graphicx}
\usepackage{amsfonts}
\newtheorem{theorem}{Theorem}[section]
\newtheorem{lemma}[theorem]{{\bf Lemma}}

\newtheorem{corollary}[theorem]{{\bf Corollary}}
\newtheorem{remark}[theorem]{{\bf Remark}}

\newtheorem{definition}{Definition}[section]
\numberwithin{equation}{section}
\newenvironment{proof}{\indent{\em Proof:}}{\quad \hfill
$\Box$\vspace*{2ex}}

\begin{document}
\setcounter{page}{1}
\begin{center}
\vspace{0.4cm} {\large{\bf On Fractional Volterra integrodifferential equations with fractional integrable impulses}} \\
\vspace{0.35cm}

Sagar T. Sutar  $^{1}$\\
sutar.sagar007@gmail.com\\
\vspace{0.35cm}
Kishor D. Kucche $^{2}$ \\
kdkucche@gmail.com \\

$^{1}$ Department of Mathematics, Vivekanand College, Kolhapur, Maharashtra, India.\\

$^{2}$ Department of Mathematics, Shivaji University, Kolhapur-416 004, Maharashtra, India.\\
\end{center}

\def\baselinestretch{1.0}\small\normalsize
\begin{abstract}
We consider a class of nonlinear fractional Volterra integrodifferential equation with fractional integrable impulses and investigate the existence and uniqueness results in the Bielecki's normed Banach spaces. Further, Bielecki--Ulam type stabilities have been demonstrated on a compact interval. A concrete example is  provided to  illustrate the outcomes we acquired.
\end{abstract}
\def\baselinestretch{1.0}
\allowdisplaybreaks

\section{Introduction}
Famous ``Ulam stability problem" of functional differential equation  raised by Ulam  \cite{Ulam} have been extended to different kinds of equations.  Wang et al.\cite{Wang2} are the first mathematicians who investigated the Ulam stability and data dependence for fractional differential equations. Thereafter, several interesting works on different Ulam type stabilities of fractional differential and integral  equations  have been reported (see for instance \cite{WangZ},\cite{Wei},\cite{WangLi},\cite{Eghbali},\cite{Wang1}).

An overview pertaining to impulsive differential equations with instantaneous impulses and its applicability in the practical dynamical systems have provided in the monograph \cite{Sam,Bain,Nt}. The impulsive differential equations with time variable impulses dealt in an interesting papers \cite{Fri,Fri1,Fri2}.  Wang et al. \cite{wang2,WangLZ} studied  existence and uniqueness of
solutions  and established generalized $\beta$-Ulam-Hyers-Rassias stability to differential equations with not instantaneous impulses in a $P\beta$-normed Banach space. 

Recently, Wang and Zang \cite {wang1} introduced a new class of impulsive differential equations of the form
\begin{equation*}
\begin{cases} 
x'(t) =  f\left(t,x(t)\right),\, t\in(s_{i},t_{i+1}],\;
i=0,1,\cdots,m,\\
x(t)  =  I^{\beta}_{t_{i},t}h_{i}(t,x(t)), ~ t\in(t_{i},s_{i}],\; i=1,2,\cdots ,m, ~\beta\in(0,1).
\end{cases}
\end{equation*}
and examined the existence and uniqueness of solutions in Bielecki's normed Banach spaces. Further, demonstrated that the corresponding equations are Bielecki-Ulam's-Hyer stable. It is seen that such sort of  formulations are adequate to depict the memory procedures of the drugs in the circulation system and the subsequent absorption for the body.

Motivated by the work of Wang and Zang \cite {wang1}, we consider the following class of nonlinear fractional  Volterra integrodifferential equation with fractional order integral impulses of the form
\begin{equation}\label{e1.1}
\begin{cases} 
^c_{\sigma_{i}}\mathcal{D}_{\tau}^\alpha x(\tau) =  f\left(\tau,x(\tau),\int_{\sigma_{i}}^{\tau}h(\sigma,x(\sigma))d\sigma\right),\, \tau\in(\sigma_{i},\tau_{i+1}],\; i=0,1,\cdots,m,~\alpha\in(0,1),\\
x(\tau)  =  \mathcal{I}^{\beta}_{\tau_{i},\tau}h_{i}(\tau,x(\tau)), ~ \tau\in(\tau_{i},\sigma_{i}],\; i=1,2,\cdots ,m, ~\beta\in(0,1),
\end{cases}
\end{equation}
and research the   existence and uniqueness of solutions and examine the outcomes relating to  Bielecki--Ulam type stabilities viz. Bielecki--Ulam--Hyers and Bielecki--Ulam--Hyers--Rassias stabilities on a compact interval.
Here  $\tau_i$ and $\sigma_i$ are pre-fixed numbers satisfying $ 0=\tau_{0}=\sigma_{0}<\tau_{1}\leq \sigma_{1}\leq \tau_{2}<\cdots< \tau_{m}\leq \sigma_{m}< \tau_{m+1}=T$,~$ f:[0,T]\times \mathbb{R}\times \mathbb{R}\to\mathbb{R},\; h:[0,T]\times \mathbb{R}\to\mathbb{R}$
and for each $i=1,2\cdots m,\; h_{i}:(\tau_{i},\sigma_{i}]\times \mathbb{R}\to\mathbb{R}$  are continuous functions,  $^c_{\sigma_{i}}\mathcal{D}_{\tau}^\alpha $ is the Caputo fractional derivative of order $\alpha $ with lower terminal at $\sigma_{i}$ and $\mathcal{I}^{\beta}_{\tau_{i},\tau}$ is the Riemann-Liouville fractional integral of order $\beta$ with lower terminal $\tau_{i}$.

We comment that within this scope, the class of equations considered in the present paper  is more broad and  the outcomes acquired are the generalization of  the fundamental results obtained by Wang and Zang \cite {wang1}. We support our main results with the examples.

 In section 2, we introduce some preliminaries and auxiliary lemmas related to fractional calculus. In Section 3, we  prove  existence and uniqueness results for
\eqref{e1.1} by using Banach contraction principle via Bielecki norm. In Section 4, adopting the idea 
of Wang and Zang \cite {wang1}  we examine different Bielecki–Ulam’s type stability for the problem \eqref{e1.1}. Finally, an example has been provided to illustrate the results we obtained. 

\section{Preliminaries} \label{preliminaries}
Let $J=[0,T]$. Let $C(J,\mathbb{R})=\left\lbrace x:J\to \mathbb{R} \,
;\, x \, \text{is continuous function}\right\rbrace$ be the Banach space endowed with a Bielecki norm 
$$\|x\|_{B}=\sup_{\tau\in J}\frac{|x(\tau)|}{e^{\theta \tau}},$$ 
where  $\theta>0$ is a fixed real number. Let 
\begin{gather*}
PC(J,\mathbb{R})=\Big\{x:J\to \mathbb{R} \,:\,  x\in C((\tau_{i},\tau_{i+1}],\mathbb{R}),  ~\text{both} ~ x(\tau_{i}^{+}) 
 ~\text{and}~ x(\tau_{i}^{-})~\text{exists and} \\ 
 x(\tau_{i})=x(\tau_{i}^{-}),\; i=0,1,\cdots m\Big\}.
\end{gather*}
If $PC(J,\mathbb{R})$ is endowed with the norm
$$\|x\|_{PB}=\sup_{\tau\in J}\frac{|x(\tau)|}{e^{\theta \tau}}, ~ \theta>0$$
then  $\Big(PC(J,\mathbb{R}), ~\|\cdot\|_{PB}\Big)$ is a Banach space. The space
\begin{gather*}
PC^{1}(J,\mathbb{R})=\Big\{x \in PC(J,\mathbb{R}) \,:\,  x' \in PC(J,\mathbb{R}) \Big\}.
\end{gather*}
with the norm $\|x\|_{PB'}=\max\left\lbrace \|x\|_{PB},~\|x'\|_{PB}\right\rbrace $ then $\Big(PC^{1}(J,\mathbb{R}), ~\|\cdot\|_{PB'}\Big)$ is a Banach space.

   Next, we use definitions and the results listed bellow  from fractional calculus. For more details, we refer the readers to the monographs \cite{Podlubny} 
\begin{definition}
Let $g\in C[a,T]$ with $T>a\geq 0 $ and $\beta \geq {0}$, then the Riemann-Liouville fractional integral $\mathcal{I}^{\beta}_{a,\tau}$ of order $\beta$ of a function $g$ is defined as
\begin{equation*}
 \mathcal{I}^{\beta}_{a,\tau}\; g(\tau)=\frac{1}{\Gamma(\beta)}\int_{a} ^\tau (\tau-\sigma)^{\beta -1 }g(\sigma)d\sigma ,\;\; \tau>a\geq 0,
\end{equation*}
provided the integral exists. 
\end{definition}

\begin{definition}
Let $0<\alpha\leq 1 $ then the Caputo fractional derivative $ ^c_{a}\mathcal{D}_{t}^\alpha$ of order $\alpha$ with lower terminal $a$ of a function $g\in C^{1}[a,T] $  is defined as
\begin{equation*}
^c_{a}\mathcal{D}_{\tau}^\alpha g(\tau)=\frac{1}{\Gamma(1-\alpha)}\int_{a}^\tau(\tau-\sigma)^{-\alpha}g'(\sigma)d\sigma,\;\; \tau>a\geq 0.
\end{equation*}
\end{definition}

\begin{lemma}\label{lm1}
Let $m-1<\alpha\leq m \in \mathbb{N}$ and $ g\in C^{m}[a,T]$. Then
$$
\mathcal{I}^{\alpha}_{a,\tau}[^c_{a}\mathcal{D}^{\alpha}_{\tau}g(\tau)]=g(\tau)-\sum_{k=0}^{m-1}\frac {g^{(k)}(a)}{\Gamma(k+1)}\tau^{k},\;\; \tau>a\geq 0.
$$
\end{lemma}
Following lemma plays an important role to obtain our results.
\begin{lemma}\label{Pru}\cite{Pud} 
Let $\alpha, \beta, \gamma $ and $p$ be constants such that $\alpha>0,\; p(\gamma-1)+1>0$ and $p(\beta-1)+1>0$. Then
$$\int_{0}^{\tau}(\tau^{\alpha}-\sigma^{\alpha})^{p(\beta-1)}\sigma^{p(\gamma-1)}d\sigma=\frac{\tau^{\theta}}{\alpha}\mathbb{B}\left( \frac{p(\gamma-1)+1}{\alpha},p(\beta-1)+1\right),\; where\; \tau\in \mathbb{R_{+}} $$
where $$\mathbb{B}(\xi,\sigma)=\int_{0}^{1} \sigma^{\xi-1}(1-\sigma)^{\sigma-1}d\sigma,\;\; Re(\xi)>0,\;Re(\sigma)>0$$ is a Beta function and $ \theta=p[\alpha(\beta-1)+\gamma-1]+1$
\end{lemma}

\section{Existence and Uniqueness Results}

\begin{lemma}\label{lm3.1}
A function  $x\in PC^{1}(J,R)$ is a classical solution of the problem
\begin{equation}\label{e3.1}
\begin{cases} 
^c_{\sigma_{i}}\mathcal{D}_{\tau}^\alpha x(\tau) =  f\left(\tau,x(\tau),\int_{\sigma_{i}}^{\tau}h(\sigma,x(\sigma))d\sigma\right),\, \tau\in(\sigma_{i},\tau_{i+1}],\; i=0,1,\cdots,m,~\alpha\in(0,1)\\
x(\tau)  =  \mathcal{I}^{\beta}_{\tau_{i},\tau}h_{i}(\tau,x(\tau)), ~ \tau\in(\tau_{i},\sigma_{i}],\; i=1,2,\cdots ,m, ~\beta\in(0,1)\\
 x(0)=x_{0}
\end{cases}
\end{equation}
if $x$ satisfies the  fractional Volterra integral equation
\begin{equation*}
x(\tau) =
\begin{cases}
x_{0},\;\tau=0,\\
\mathcal{I}^{\beta}_{\tau_{i},\tau}h_{i}(\tau,x(\tau)),\;\tau\in(\tau_{i},\sigma_{i}],\; i=1,\cdots ,m,   \\
x_{0}+\mathcal{I}^{\alpha}_{0,\tau}f\left(\tau,x(\tau),\in\tau_{0}^{\tau}h(\sigma,x(\sigma))d\sigma\right),\;  \tau\in (0,\tau_{1}],\\
\mathcal{I}^{\beta}_{\tau_{i},\sigma_{i}}h_{i}(\sigma_{i},x(\sigma_{i}))+\mathcal{I}^{\alpha}_{\sigma_{i},\tau}f\left(\tau,x(\tau),\int_{\sigma_{i}}^{\tau}h(\sigma,x(\sigma))d\sigma\right),\tau\in (\sigma_{i},\tau_{i+1}], i=1,\cdots m.
\end{cases}
\end{equation*}
\end{lemma}
\begin{proof}
For $i=0$,  on operating Riemann-Liouville fractional integral operator  $\mathcal{I}^{\alpha}_{\sigma_{i},\tau}$ on both sides of fractional differential equation \eqref{e3.1}, we get
$$
\mathcal{I}^{\alpha}_{0,\tau}[^c_{0}\mathcal{D}_{\tau}^\alpha x(\tau)] =\mathcal{I}^{\alpha}_{0,\tau}f\left(\tau,x(\tau),\int_{0}^{\tau}h(\sigma,x(\sigma))d\sigma\right),\; \tau\in (0,\tau_{1}]
$$
As $0<\alpha<1$, in view of lemma \ref{lm1}, we obtain
$$x(\tau)-x(0)=\mathcal{I}^{\alpha}_{0,\tau}f\left(\tau,x(\tau),\int_{0}^{\tau}h(\sigma,x(\sigma))d\sigma\right),\; \tau\in (0,\tau_{1}]$$
Therefore
\begin{equation*}
x(\tau)=x_{0}+\mathcal{I}^{\alpha}_{0,\tau}f\left(\tau,x(\tau),\int_{0}^{\tau}h(\sigma,x(\sigma))d\sigma\right),\;\; \tau\in (0,\tau_{1}]
\end{equation*}
Similarly, for each $ i ~(i=1,2,\cdots m)$, operating    $\mathcal{I}^{\alpha}_{\sigma_{i},\tau}$ on both sides of \eqref{e3.1}, we get 
\begin{equation*}
x(\tau)=x(\sigma_{i})+\mathcal{I}^{\alpha}_{\sigma_{i},\tau}f\left(\tau,x(\tau),\int_{\sigma_{i}}^{\tau}h(\sigma,x(\sigma))d\sigma\right),\;\; \tau\in (\sigma_{i},\tau_{i+1}]
\end{equation*}
But from equation \eqref{e3.1}, we have
$$x(\sigma_{i})=\left[ \mathcal{I}^{\beta}_{\tau_{i},\tau}h_{i}(\tau,x(\tau))\right]_{\tau=\sigma_{i}} =\mathcal{I}^{\beta}_{\tau_{i},\sigma_{i}}h_{i}(\sigma_{i},x(\sigma_{i})),~ \; i=1,2,\cdots ,m$$
Therefore
\begin{equation*}
x(\tau)=\mathcal{I}^{\beta}_{\tau_{i},\sigma_{i}}h_{i}(\sigma_{i},x(\sigma_{i}))+\mathcal{I}^{\alpha}_{\sigma_{i},\tau}f\left(\tau,x(\tau),\int_{\sigma_{i}}^{\tau}h(\sigma,x(\sigma))d\sigma\right),\;\; \tau\in (\sigma_{i},\tau_{i+1}].
\end{equation*}
\end{proof}
We list the following hypotheses in order to establish our main results.
\begin{itemize}
\item[(H1)] The function $f\in C(J\times\mathbb{R}\times\mathbb{R},\mathbb{R})$ satisfies the Lipschitz  condition 
$$ |f(\tau,x,y)-f(\tau,\bar{x},\bar{y})|\leq M_{f}|x-\bar{x}|+N_{f}|y-\bar{y}|,~\tau\in J ;\, \,  x,\bar{x},y,\bar{y}\in\mathbb R,$$
where $M_{f}>0,\;N_{f}>0.$

\item[(H2)] The function $h\in C(J\times\mathbb{R},\mathbb{R})$ satisfies the Lipschitz  condition 
$$ |h(\tau,x)-h(\tau,\bar{x})|\leq K_{h}|x-\bar{x}|,~ \tau\in J ;\, \,  x,\bar{x}\in\mathbb R,$$
where $K_{h}>0.$
\item[(H3)]For each $i=1,2,\cdots m $; $h_{i}\in C((\sigma_{i},\tau_{i}]\times\mathbb{R},\mathbb{R})$ and satisfies Lipschitz condition
$$ |h_{i}(\tau,x)-h_{i}(\tau,\bar{x})|\leq L_{h_{i}}|x-\bar{x}|, \text{for each } \tau\in (\sigma_{i},\tau_{i}] ;\, \,  x,\bar{x}\in\mathbb R,$$
where $L_{h_{i}}>0.$
\end{itemize}

\begin{theorem}\label{Thm1}
Assume that hypotheses $ (H1)$-$(H3)$ holds. Then the problem \eqref{e3.1} has a unique solution, provided that $\alpha,\beta\in\left(\frac{1}{2},1\right).$
\end{theorem}
\begin{proof}
Define an operator $T:PC(J,\mathbb{R})\to PC(J,\mathbb{R})$ by 
\begin{equation*}
(Tx)(\tau) =
\begin{cases}
x_{0},\text{if  $\tau=0,$}\\
\mathcal{I}^{\beta}_{\tau_{i},\tau}h_{i}(\tau,x(\tau)),\text{ if $ \tau\in(\tau_{i},\sigma_{i}],\; i=1,\cdots ,m ,$ } \\
x_{0}+\mathcal{I}^{\alpha}_{0,\tau}f\left(\tau,x(\tau),\int_{0}^{\tau}h(\sigma,x(\sigma))d\sigma\right),
\text{ if $\tau\in (0,\tau_{1}],$}\\
\mathcal{I}^{\beta}_{\tau_{i},\sigma_{i}}h_{i}(\sigma_{i},x(\sigma_{i}))+\mathcal{I}^{\alpha}_{\sigma_{i},\tau}f\left(\tau,x(\tau),\int_{\sigma_{i}}^{\tau}h(\sigma,x(\sigma))d\sigma\right),\text{ if  $\tau\in (\sigma_{i},\tau_{i+1}]$}, ~i=1,\cdots m.\\
\end{cases}
\end{equation*}
We shall show that the operator $T$ is  contraction with respect to Bielecki norm. Let ${x,y \in PC(J,\mathbb{R})}$ and $\tau\in (\tau_{i},\sigma_{i}],\;i=1,2,\cdots m$, then using hypothesis $(H3)$, we have
\begin{align*}
|& \mathcal{I}^{\beta}_{\tau_{i},\tau}h_{i}(\tau,x(\tau))-\mathcal{I}^{\beta}_{\tau_{i},\tau}h_{i}(\tau,y(\tau))|\\
&=\left| \dfrac{1}{\Gamma(\beta)}\int_{\tau_{i}}^{\tau}(\tau-\sigma)^{\beta-1}h_{i}(\sigma,x(\sigma))d\sigma-\dfrac{1}{\Gamma(\beta)}\int_{\tau_{i}}^{\tau}(\tau-\sigma)^{\beta-1}h_{i}(\sigma,y(\sigma))d\sigma\right| \\
&\leq \dfrac{1}{\Gamma(\beta)}\int_{\tau_{i}}^{\tau}(\tau-\sigma)^{\beta-1}\left|h_{i}(\sigma,x(\sigma))-h_{i}(\sigma,y(\sigma))\right|d\sigma\\
&\leq \dfrac{L_{h_{i}}}{\Gamma(\beta)}\int_{\tau_{i}}^{\tau}(\tau-\sigma)^{\beta-1}\left|x(\sigma)-y(\sigma)\right|d\sigma\\
&\leq\dfrac{L_{h_{i}}}{\Gamma(\beta)}\int_{\tau_{i}}^{\tau}(\tau-\sigma)^{\beta-1}e^{\theta \sigma}\left( \sup_{\sigma\in (\tau_{i},\sigma_{i}]}e^{-\theta \sigma}|x(\sigma)-y(\sigma)|\right)d\sigma
\\
&\leq \dfrac{L_{h_{i}}}{\Gamma(\beta)}\|x-y\|_{PB}\int_{\tau_{i}}^{\tau}(\tau-\sigma)^{\beta-1}e^{\theta \sigma}d\sigma
\end{align*}
By Holders inequality,  for $\beta\in (\frac{1}{2},1)$ and $\tau\in (\tau_{i},\sigma_{i}]$, we have
\begin{align}\label{Hol}
\int_{\tau_{i}}^{\tau}(\tau-\sigma)^{\beta-1}e^{\theta \sigma}d\sigma
&\leq\left( \int_{\tau_{i}}^{\tau}(\tau-\sigma)^{2(\beta-1)}d\sigma\right) ^{\frac{1}{2}} \left( \int_{\tau_{i}}^{\tau} e^{2\theta \sigma}d\sigma\right) ^{\frac{1}{2}} \nonumber\\
& = \left(\dfrac{(\tau-\tau_{i})^{2\beta-1}}{2\beta-1}\right)^{\frac{1}{2}}\left(\dfrac{e^{2\theta \tau}-e^{2\theta \tau_{i}}}{2\theta} \right)^\frac{1}{2} \nonumber\\
&\leq  \dfrac{(\sigma_{i}-\tau_{i})^{\beta-\frac{1}{2}}}{\left( 2\beta-1\right)^\frac{1}{2}} \left( \dfrac{e^{\theta \tau}-e^{\theta \tau_{i}}}{(2\theta)^\frac{1}{2}}\right) \nonumber\\
&\leq  \dfrac{(\sigma_{i}-\tau_{i})^{\beta-\frac{1}{2}}}{\left( 2\beta-1\right)^\frac{1}{2}(2\theta)^\frac{1}{2}}e^{\theta \tau}
\end{align}
Thus for any $ x,y \in PC(J,R)$  and $\tau\in (\tau_{i},\sigma_{i}],\; i=1,2,\cdots m, $  we have
\begin{equation}\label{ie3.3}
\left| \mathcal{I}^{\beta}_{\tau_{i},\tau}h_{i}(\tau,x(\tau))-\mathcal{I}^{\beta}_{\tau_{i},\tau}h_{i}(\tau,y(\tau))\right|\leq \dfrac{L_{h_{i}}}{\Gamma(\beta)} \dfrac{(\sigma_{i}-\tau_{i})^{\beta-\frac{1}{2}}}{\left( 2\beta-1\right)^\frac{1}{2}(2\theta)^\frac{1}{2}}\,e^{\theta \tau}\,\|x-y\|_{PB}.
\end{equation}
Now,  for any $x,y \in PC(J,R)$ and $\tau\in (\sigma_{i},\tau_{i+1}],i=1,2,\cdots m,$ by using hypotheses $(H1)-(H2)$, we obtain
 \begin{align} \label{in3.5}
&\left|\mathcal{I}^{\alpha}_{\sigma_{i},\tau}f\left(\tau,x(\tau),\int_{\sigma_{i}}^{\tau}h(\sigma,x(\sigma))d\sigma\right)-\mathcal{I}^{\alpha}_{\sigma_{i},\tau}f\left(\tau,y(\tau),\int_{\sigma_{i}}^{\tau}h(\sigma,y(\sigma))d\sigma\right) \right| \nonumber\\
& \leq \mathcal{I}^{\alpha}_{\sigma_{i},\tau}\left|f\left(\tau,x(\tau),\mathcal{I}^{1}_{\sigma_{i},\tau}h(\tau,x(\tau)\right)-f\left(\tau,y(\tau),\mathcal{I}^{1}_{\sigma_{i},\tau}h(\tau,y(\tau)\right) \right| \nonumber\\
&\leq \mathcal{I}^{\alpha}_{\sigma_{i},\tau}\left(M_{f}|x(\tau)-y(\tau)|+N_{f}|\mathcal{I}^{1}_{\sigma_{i},\tau}h(\tau,y(\tau))-\mathcal{I}^{1}_{\sigma_{i},\tau}h(\tau,y(\tau))|\right)\nonumber\\
&\leq \mathcal{I}^{\alpha}_{\sigma_{i},\tau}\left( M_{f}|x(\tau)-y(\tau)|+N_{f}\mathcal{I}^{1}_{\sigma_{i},\tau} |h(\tau,x(\tau))-h(\tau,y(\tau))|\right) \nonumber\\
&\leq \mathcal{I}^{\alpha}_{\sigma_{i},\tau}\left( M_{f}|x(\tau)-y(\tau)|+N_{f}\mathcal{I}^{1}_{\sigma_{i},\tau}K_{h} |x(\tau)-y(\tau))|\right) \nonumber\\
&= M_{f}\mathcal{I}^{\alpha}_{\sigma_{i},\tau}|x(\tau)-y(\tau)|+N_{f}K_{h}\mathcal{I}^{\alpha}_{\sigma_{i},\tau}\left( \mathcal{I}^{1}_{\sigma_{i},\tau} |x(\tau)-y(\tau)|\right) \nonumber \\
&= M_{f}\mathcal{I}^{\alpha}_{\sigma_{i},\tau}|x(\tau)-y(\tau)|+N_{f}K_{h}\mathcal{I}^{\alpha+1}_{\sigma_{i},\tau} |x(\tau)-y(\tau)| \nonumber\\
&=\frac{M_{f}}{\Gamma(\alpha)}\int_{\sigma_{i}}^{\tau}(\tau-\sigma)^{\alpha-1}|x(\sigma)-y(\sigma)|d\sigma+\frac{N_{f}K_{h}}{{\Gamma(\alpha+1)}} \int_{\sigma_{i}}^{\tau}(\tau-\sigma)^{\alpha}|x(\sigma)-y(\sigma)|d\sigma \nonumber\\
&= \frac{M_{f}}{\Gamma(\alpha)}\int_{\sigma_{i}}^{\tau}(\tau-\sigma)^{\alpha-1}e^{\theta \sigma}\left( \sup_{\sigma\in (\sigma_{i},\tau_{i+1}]}e^{-\theta \sigma}|x(\sigma)-y(\sigma)|\right)d\sigma \nonumber\\
&\qquad+\frac{N_{f}K_{h}}{{\Gamma(\alpha+1)}} \int_{\sigma_{i}}^{\tau}(\tau-\sigma)^{\alpha}|x(\sigma)-y(\sigma)|e^{\theta \sigma}\left( \sup_{\sigma\in (\sigma_{i},\tau_{i+1}]}e^{-\theta \sigma}|x(\sigma)-y(\sigma)|\right)d\sigma \nonumber\\
&\leq\frac{M_{f}}{\Gamma(\alpha)}\|x-y\|_{PB}\int_{\sigma_{i}}^{\tau}(\tau-\sigma)^{\alpha-1}e^{\theta \sigma}d\sigma+\frac{N_{f}K_{h}}{{\Gamma(\alpha+1)}}\|x-y\|_{PB} \int_{\sigma_{i}}^{\tau}(\tau-\sigma)^{\alpha}e^{\theta \sigma}d\sigma.
\end{align}
Since $\alpha\in (\frac{1}{2},1)$, from the inequality \eqref{Hol}, for any $\tau\in(\sigma_{i},\tau_{i+1}],\; i= 1,2,\cdots ,m$, we have
\begin{equation}\label{in3.6}
\int_{\sigma_{i}}^{\tau}(\tau-\sigma)^{\alpha-1}e^{\theta \sigma}d\sigma\leq \frac{(\tau-\sigma_{i})^{\alpha-\frac{1}{2}}}{\left( 2\alpha-1\right)^\frac{1}{2}(2\theta)^\frac{1}{2}}e^{\theta \tau}\leq \frac{(\tau_{i+1}-\sigma_{i})^{\alpha-\frac{1}{2}}}{\left( 2\alpha-1\right)^\frac{1}{2}(2\theta)^\frac{1}{2}}e^{\theta \tau}
\end{equation}
By replacing $\alpha$ by $(\alpha+1)$ in the above inequality we get
\begin{equation}\label{in3.7}
\int_{\sigma_{i}}^{t}(t-\sigma)^{\alpha}e^{\theta \sigma}d\sigma\leq \frac{(t_{i+1}-\sigma_{i})^{\alpha+\frac{1}{2}}}{\left( 2\alpha+1\right)^\frac{1}{2}(2\theta)^\frac{1}{2}}e^{\theta t},\;\;t\in(\sigma_{i},t_{i+1}],\; i= 1,2,\cdots ,m
\end{equation}
Using the inequalities \eqref{in3.6} and \eqref{in3.7}, the inequality \eqref{in3.5} takes the form
\begin{align}\label{in3.8}
&\left|  \mathcal{I}^{\alpha}_{\sigma_{i},\tau}f\left(\tau,x(\tau),\int_{\sigma_{i}}^{\tau}h(\sigma,x(\sigma)\right)d\sigma)-\mathcal{I}^{\alpha}_{\sigma_{i},\tau}f\left(\tau,y(\tau),\int_{\sigma_{i}}^{\tau}h(\sigma,y(\sigma))d\sigma\right) \right|\nonumber\\
& \leq\left[ \frac{M_{f}}{\Gamma(\alpha)}\frac{(\tau_{i+1}-\sigma_{i})^{\alpha-\frac{1}{2}}}{\left( 2\alpha-1\right)^\frac{1}{2}(2\theta)^\frac{1}{2}}+\frac{N_{f}K_{h}}{{\Gamma(\alpha+1)}} \frac{(\tau_{i+1}-\sigma_{i})^{\alpha+\frac{1}{2}}}{\left( 2\alpha+1\right)^\frac{1}{2}(2\theta)^\frac{1}{2}}\right]\,e^{\theta \tau}\,\|x-y\|_{PB} 
\end{align}
From the definition of an operator $T$ and hypothesis (H1)-(H3), for any $x,y\in PC(J,\mathbb{R})$, we obtain that\\
\noindent  Case~1: For $\tau\in (\tau_{i},\sigma_{i}],\; i=1,2,\cdots m,$ from the inequality \eqref{ie3.3}, we have
\begin{align*}
| T  x(\tau)-Ty(\tau)|e^{-\theta \tau}&\leq\left| \mathcal{I}^{\beta}_{\tau_{i},\tau}h_{i}(\tau,x(\tau))-\mathcal{I}^{\beta}_{\tau_{i},\tau}h_{i}(\tau,y(\tau))\right|e^{-\theta \tau}\\
& \leq \dfrac{L_{h_{i}}}{\Gamma(\beta)}\dfrac{(\sigma_{i}-\tau_{i})^{\beta-\frac{1}{2}}}{\left( 2\beta-1\right)^\frac{1}{2}(2\theta)^\frac{1}{2}}\|x-y\|_{PB}.
\end{align*}
\noindent  Case~2: For $\tau\in (0,\tau_{1}]$, on similar line of \eqref{in3.8}, we have
\begin{align*}
&| T  x(\tau)-Ty(\tau)|e^{-\theta \tau}\\
& =\left|  \mathcal{I}^{\alpha}_{0,\tau}f\left(\tau,x(\tau),\int_{0}^{\tau}h(\sigma,x(\sigma))d\sigma\right)-\mathcal{I}^{\alpha}_{0,\tau}f\left(\tau,y(\tau),\int_{0}^{\tau}h(\sigma,y(\sigma))d\sigma\right) \right|e^{-\theta \tau}\\
&\leq\left[ \frac{M_{f}}{\Gamma(\alpha)}\frac{\tau_{1}^{\alpha-\frac{1}{2}}}{\left( 2\alpha-1\right)^\frac{1}{2}(2\theta)^\frac{1}{2}}+\frac{N_{f}K_{h}}{{\Gamma(\alpha+1)}} \frac{\tau_{1}^{\alpha+\frac{1}{2}}}{\left( 2\alpha+1\right)^\frac{1}{2}(2\theta)^\frac{1}{2}}\right]\|x-y\|_{PB} 
\end{align*}
\noindent  Case~3: For $\tau\in (\sigma_{i},\tau_{i+1}],\; i=1,2,\cdots m$, using the inequalities \eqref{ie3.3} and \eqref{in3.8}, we have 
\begin{align*}
&| T  x(\tau)-Ty(\tau)|e^{-\theta \tau}\\
&\;= \left| \mathcal{I}^{\beta}_{\tau_{i},\sigma_{i}}h_{i}(\sigma_{i},x(\sigma_{i}))-\mathcal{I}^{\beta}_{\tau_{i},\sigma_{i}}h_{i}(\sigma_{i},y(\sigma_{i}))\right|e^{-\theta \tau}\\
&\quad+\left| \mathcal{I}^{\alpha}_{\sigma_{i},\tau}f(\tau,x(\tau),\int_{\sigma_{i}}^{\tau}h(\sigma,x(\sigma))d\sigma)-\mathcal{I}^{\alpha}_{\sigma_{i},t}f(\tau,y(\tau),\int_{\sigma_{i}}^{\tau}h(\sigma,y(\sigma))d\sigma)\right|e^{-\theta \tau}  \\
&\leq \dfrac{L_{h_{i}}}{\Gamma(\beta)} \dfrac{(\sigma_{i}-\tau_{i})^{\beta-\frac{1}{2}}}{\left( 2\beta-1\right)^\frac{1}{2}(2\theta)^\frac{1}{2}}e^{\theta (\sigma_{i}-\tau)}\|x-y\|_{PB}\\
&\quad +\left[ \frac{M_{f}}{\Gamma(\alpha)}\frac{(\tau_{i+1}-\sigma_{i})^{\alpha-\frac{1}{2}}}{\left( 2\alpha-1\right)^\frac{1}{2}(2\theta)^\frac{1}{2}}+\frac{N_{f}K_{h}}{{\Gamma(\alpha+1)}} \frac{(\tau_{i+1}-\sigma_{i})^{\alpha+\frac{1}{2}}}{\left( 2\alpha+1\right)^\frac{1}{2}(2\theta)^\frac{1}{2}}\right]\|x-y\|_{PB} \\
&\leq \left( \frac{L_{h_{i}}}{\Gamma(\beta)} \frac{(\sigma_{i}-\tau_{i})^{\beta-\frac{1}{2}}}{\left( 2\beta-1\right)^\frac{1}{2}(2\theta)^\frac{1}{2}} +\frac{M_{f}}{\Gamma(\alpha)}\frac{(\tau_{i+1}-\sigma_{i})^{\alpha-\frac{1}{2}}}{\left( 2\alpha-1\right)^\frac{1}{2}(2\theta)^\frac{1}{2}} \right.\\ &\qquad\qquad\left.
\qquad+\frac{N_{f}K_{h}}{{\Gamma(\alpha+1)}} \frac{(\tau_{i+1}-\sigma_{i})^{\alpha+\frac{1}{2}}}{\left( 2\alpha+1\right)^\frac{1}{2}(2\theta)^\frac{1}{2}}\right) \|x-y\|_{PB}
\end{align*}
From the cases 1-3, we can conclude that
$$\|Tx-Ty\|_{PB}\leq L \|x-y\|_{PB}, ~ \mbox{for any} ~x,y\in PC(J,\mathbb{R}), $$
where
 \begin {align*}
L=& \max\left\lbrace \frac{N_{h_{i}}}{\Gamma(\beta)} \frac{(\sigma_{i}-\tau_{i})^{\beta-\frac{1}{2}}}{\sqrt{(2\beta-1)}\sqrt{2\theta}}+\frac{M_{f}}{\Gamma(\alpha)}\frac{(\tau_{i+1}-\sigma_{i})^{\alpha-\frac{1}{2}}}{\sqrt{(2\alpha-1)}\sqrt{2\theta}}  
\right.\\ &\qquad\qquad\left.
\qquad+\frac{N_{f}N_{h_{i}}}{{\Gamma(\alpha+1)}} \frac{(\tau_{i+1}-\sigma_{i})^{\alpha}}{\sqrt{2\alpha}\sqrt{2\theta}} ; i=0,1,\cdots m\right\rbrace
\end{align*}
Choose sufficiently large value of $\theta$ so that $L<1$. Thus, $T$ becomes  a contraction mapping and has a unique fixed point due to Banach contraction principle. This fixed point of $T$ act as unique solution of the problem \eqref{e3.1}.
\end{proof}

Next, our aim is to extend the restriction of $\alpha,\beta \in (\frac{1}{2},1)$  to $\alpha,\beta \in (0,1)$. For this purpose we use the following variants of the hypotheses  $(H2)$ and $(H3)$:
\begin{itemize}
\item[$(\tilde{H}1)$] The function $f\in C(J\times\mathbb{R}\times\mathbb{R},\mathbb{R})$ satisfies the Lipschitz  condition
$$ |f(\tau,x,y)-f(\tau,\bar{x},\bar{y})|\leq  M_{f}\tau^{\gamma_{f}}|x-\bar{x}|+N_{f}|y-\bar{y}|,\, \tau\in [0,T] ;\, \,  x,\bar{x},y,\bar{y}\in\mathbb R,$$ where $M_{f}, N_{f} >0$ and $\gamma_{f}>-\alpha.$ 
\item[$(\tilde{H}3)$]For each $i=1,2,\cdots m $; $h_{i}\in C((\sigma_{i},\tau_{i}]\times\mathbb{R},\mathbb{R})$ and satisfies Lipschitz condition
$$ |h_{i}(\tau,x)-h_{i}(\tau,\bar{x})|\leq \tau^{\gamma}N_{h_{i}}|x-\bar{x}|, \tau\in [\tau_{i},\sigma_{i}] ;\, \,  x,\bar{x}\in\mathbb R,$$ where $ N_{h_{i}}>0$ and $\gamma>-\beta.$ 
\end{itemize}
\begin{theorem}\label{Thm2}
Assume that hypotheses $(\tilde{H}1)$, $(H2)$  and $(\tilde{H}3)$ holds. Then the problem \eqref{e3.1} has a unique solution.
\end{theorem}
\begin{proof} Consider the operator $T$ defined  in the Theorem \ref{Thm1} .
By $(\tilde{H}1)$, $\gamma_{f}>-\alpha=(1-\alpha)-1$ where $\alpha\in (0,1)$. Choose $\sigma>1$ such that $\sigma \,\gamma_{f}> \sigma(1-\alpha)-1 $ and $\sigma(\alpha-1)+1 >0.$ Define $\sigma^{*}=\dfrac{\sigma}{\sigma-1}$. Then  $\sigma+\sigma^{*}=1.$
By $(\tilde{H}3)$, for any $\tau\in (\sigma_{i},\tau_{i+1}]$, we have
 \begin{align}\label{uh3.17}
&\left|\mathcal{I}^{\alpha}_{\sigma_{i},\tau}f\left(\tau,x(\tau),\int_{\sigma_{i}}^{\tau}h(\sigma,x(\sigma))d\sigma\right)-\mathcal{I}^{\alpha}_{\sigma_{i},\tau}f\left(\tau,y(\tau),\int_{\sigma_{i}}^{\tau}h(\sigma,y(\sigma)\right)d\sigma) \right|\nonumber\\
& \leq \mathcal{I}^{\alpha}_{\sigma_{i},\tau}\left|f\left(\tau,x(\tau),\mathcal{I}^{1}_{\sigma_{i},\tau}h(\tau,x(\tau)\right)-f\left(\tau,y(\tau),\mathcal{I}^{1}_{\sigma_{i},\tau}h(\tau,y(\tau)\right) \right|\nonumber\\
&\leq \mathcal{I}^{\alpha}_{\sigma_{i},\tau}\left(M_{f}\tau^{\gamma_{f}}|x(t)-y(\tau)|+N_{f}|\mathcal{I}^{1}_{\sigma_{i},\tau}h(\tau,y(\tau))-\mathcal{I}^{1}_{\sigma_{i},\tau}h(\tau,y(\tau))|\right)\nonumber\\
&\leq \mathcal{I}^{\alpha}_{\sigma_{i},\tau}\left( M_{f}\tau^{\gamma_{f}}|x(\tau)-y(\tau)|+N_{f}\mathcal{I}^{1}_{\sigma_{i},\tau} |h(\tau,x(\tau))-h(\tau,y(\tau))|\right) \nonumber\\
&\leq \mathcal{I}^{\alpha}_{\sigma_{i},\tau}\left( M_{f}\tau^{\gamma_{f}}|x(\tau)-y(\tau)|+N_{f}\mathcal{I}^{1}_{\sigma_{i},\tau}K_{h} |x(\tau)-y(\tau)|\right) \nonumber\\
&= M_{f}\mathcal{I}^{\alpha}_{\sigma_{i},\tau}\tau^{\gamma_{f}}|x(\tau)-y(\tau)|+N_{f}K_{h}\mathcal{I}^{\alpha}_{\sigma_{i},\tau}\left( \mathcal{I}^{1}_{\sigma_{i},\tau} |x(\tau)-y(\tau)|\right)\nonumber  \\
&= M_{f}\mathcal{I}^{\alpha}_{\sigma_{i},\tau}\left( \tau^{\gamma_{f}}|x(\tau)-y(\tau)|\right) +N_{f}K_{h}\mathcal{I}^{\alpha+1}_{\sigma_{i},\tau} \left( |x(\tau)-y(\tau)|\right) \nonumber\\
&=\frac{M_{f}}{\Gamma(\alpha)}\int_{\sigma_{i}}^{\tau}(\tau-\sigma)^{\alpha-1}\sigma^{\gamma_{f}}|x(\sigma)-y(\sigma)|d\sigma+\frac{N_{f}K_{h}}{{\Gamma(\alpha+1)}} \int_{\sigma_{i}}^{\tau}(\tau-\sigma)^{\alpha}|x(\sigma)-y(\sigma)|d\sigma\nonumber\\
&\leq \frac{M_{f}}{\Gamma(\alpha)}\int_{\sigma_{i}}^{\tau}(\tau-\sigma)^{\alpha-1}\sigma^{\gamma_{f}}e^{\theta \sigma}\left( \sup_{\sigma\in (\sigma_{i},\tau_{i+1}]}e^{-\theta \sigma}|x(\sigma)-y(\sigma)|\right)d\sigma\nonumber\\
&\qquad+\frac{N_{f}K_{h}}{{\Gamma(\alpha+1)}} \int_{\sigma_{i}}^{\tau}(\tau-\sigma)^{\alpha}|x(\sigma)-y(\sigma)|e^{\theta \sigma}\left( \sup_{\sigma\in (\sigma_{i},\tau_{i+1}]}e^{-\theta \sigma}|x(\sigma)-y(\sigma)|\right)d\sigma \nonumber\\
&\leq\frac{M_{f}}{\Gamma(\alpha)}\|x-y\|_{PB}\int_{\sigma_{i}}^{\tau}(\tau-\sigma)^{\alpha-1}\sigma^{\gamma_{f}}e^{\theta \sigma}d\sigma+\frac{N_{f}K_{h}}{{\Gamma(\alpha+1)}}\|x-y\|_{PB} \int_{\sigma_{i}}^{\tau}(\tau-\sigma)^{\alpha}e^{\theta \sigma}d\sigma
\end{align}
Using Holders inequality, we get
\begin{align*}
\int_{\sigma_{i}}^{\tau}(\tau-\sigma)^{\alpha-1}\sigma^{\gamma_{f}}e^{\theta \sigma}d\sigma\leq\left( \int_{\sigma_{i}}^{\tau}(\tau-\sigma)^{\sigma(\alpha-1)}\sigma^{\sigma\gamma_{f}}d\sigma\right)^{\frac{1}{\sigma}}\left( \int_{\sigma_{i}}^{\tau}e^{\theta \sigma^{*} \sigma}d\sigma\right)^{\frac{1}{\sigma^{*}}}  
\end{align*}
But
\begin{equation*}
\left( \int_{\sigma_{i}}^{\tau}e^{\theta \sigma^{*} \sigma}d\sigma\right)^{\frac{1}{\sigma^{*}}}=\left( \frac{e^{\theta \sigma^{*} \tau}-e^{\theta \sigma^{*} \sigma_{i}}}{\theta \sigma^{*} }\right)^{\frac{1}{\sigma^{*}}}\leq \left( \frac{e^{\theta \sigma^{*} \tau}}{\theta \sigma^{*} }\right) ^{\frac{1}{\sigma^{*}}}= \frac{e^{\theta \tau}}{(\theta \sigma^{*})^{\frac{1}{\sigma^{*}}}}
\end{equation*}
By  lemma \ref{Pru}, 
\begin{align*}
 \left( \int_{\sigma_{i}}^{\tau}(\tau-\sigma)^{\sigma(\alpha-1)}\sigma^{\sigma\gamma_{f}}d\sigma\right)^{\frac{1}{\sigma}}&\leq \left( \int_{0}^{\tau}(\tau-\sigma)^{\sigma(\alpha-1)}\sigma^{\sigma\gamma_{f}}d\sigma\right)^{\frac{1}{\sigma}}\nonumber\\
 &=\left( {\tau^{\sigma(\alpha-1)+\sigma\gamma_{f}+1}}\mathbb{B}(\sigma\gamma_{f}+1,\sigma(\alpha-1)+1) \right)^{\frac{1}{\sigma}}\nonumber\\
 & \leq \omega_{1}, 
\end{align*}
where $\omega_{1}=\left( {T^{\sigma(\alpha-1)+\sigma\gamma_{f}+1}}\mathbb{B}(\sigma\gamma_{f}+1,\sigma(\alpha-1)+1) \right)^{\frac{1}{\sigma}}$. Therefore
\begin{equation}\label{e3.17}
\int_{\sigma_{i}}^{\tau}(\tau-\sigma)^{\alpha-1}\sigma^{\gamma_{f}}e^{\theta \sigma}d\sigma\leq \omega_{1}\frac{e^{\theta \tau}}{(\theta \sigma^{*})^{\frac{1}{\sigma^{*}}}}
\end{equation}
Using inequations \eqref{in3.7} and \eqref{e3.17}, from \eqref{uh3.17} we get
\begin{align}\label{e3.19}
&\left|\mathcal{I}^{\alpha}_{\sigma_{i},\tau}f\left(\tau,x(\tau),\int_{\sigma_{i}}^{\tau}h(\sigma,x(\sigma)\right)d\sigma)-\mathcal{I}^{\alpha}_{\sigma_{i},\tau}f\left(\tau,y(\tau),\int_{\sigma_{i}}^{\tau}h(\sigma,y(\sigma))d\sigma\right)\right|e^{-\theta \tau}\nonumber\\
&\leq \left( \omega_{1}\frac{M_{f}}{\Gamma(\alpha)(\theta \sigma^{*})^{\frac{1}{\sigma^{*}}}} +\frac{N_{f}K_{h}}{{\Gamma(\alpha+1)}}\frac{(\tau_{i+1}-\sigma_{i})^{\alpha+\frac{1}{2}}}{\left( 2\alpha+1\right)^\frac{1}{2}(2\theta)^\frac{1}{2}}\right) \|x-y\|_{PB}
\end{align}
On the similar line by $(\tilde{H}3)$, as discussed above, we can chose $\sigma_{1}>1$ such that $\sigma_{1}\gamma>\sigma_{1}(1-\beta)-1$ and $\sigma_{1}(\beta-1)+1>0.$ Then for  $\sigma_{1}^{*}=\frac{\sigma_{1}}{\sigma_{1}-1}$ we have $\sigma_{1}+\sigma_{1}^{*}=1.$
By $(\tilde{H}3)$ and the Holders inequality, for any $\tau\in (\tau_{i},\sigma_{i}]$,
\begin{align*}\label{uh3.26}
&|\mathcal{I}^{\beta}_{\tau_{i},\tau}h_{i}(\tau,x(\tau))-\mathcal{I}^{\beta}_{\tau_{i},\tau}h_{i}(\tau,y(\tau))|\\
&\leq \dfrac{L_{h_{i}}}{\Gamma(\beta)}\int_{\tau_{i}}^{\tau}(\tau-\sigma)^{\beta-1}\sigma^{\gamma}\left|x(\sigma)-y(\sigma)\right|d\sigma\\
&\leq \dfrac{L_{h_{i}}}{\Gamma(\beta)}\|x-y\|_{PB}\int_{\tau_{i}}^{\tau}(\tau-\sigma)^{\beta-1}\sigma^{\gamma}e^{\theta \sigma}d\sigma \\
&\leq \dfrac{L_{h_{i}}}{\Gamma(\beta)}\|x-y\|_{PB}\left( \int_{\tau_{i}}^{\tau}(\tau-\sigma)^{\sigma_{1}(\beta-1)}\sigma^{\sigma_{1}\gamma}d\sigma\right)^{\frac{1}{\sigma_{1}}}\left( \int_{\tau_{i}}^{\tau}e^{\theta \sigma_{1}^{*} \sigma}d\sigma\right)^{\frac{1}{\sigma_{1}^{*}}}  \\
&\leq \dfrac{L_{h_{i}}}{\Gamma(\beta)}\|x-y\|_{PB}\left( \int_{\tau_{i}}^{\tau}(\tau-\sigma)^{\sigma_{1}(\beta-1)}\sigma^{\sigma_{1}\gamma}d\sigma\right)^{\frac{1}{\sigma_{1}}}\left( \int_{\tau_{i}}^{\tau}e^{\theta \sigma_{1}^{*} \sigma}d\sigma\right)^{\frac{1}{\sigma_{1}^{*}}}  \\
&\leq \dfrac{L_{h_{i}}}{\Gamma(\beta)}\|x-y\|_{PB}\left( \int_{0}^{\tau}(\tau-\sigma)^{\sigma_{1}(\beta-1)}\sigma^{\sigma_{1}\gamma}d\sigma\right)^{\frac{1}{\sigma_{1}}} \frac{e^{\theta \tau}}{(\theta \sigma_{1}^{*})^{\frac{1}{\sigma_{1}^{*}}}}\\
&\leq \dfrac{L_{h_{i}}\,\omega_{2} }{\Gamma(\beta)\,(\theta \sigma_{1}^{*})^{\frac{1}{\sigma_{1}^{*}}}}\,e^{\theta \tau}\,\|x-y\|_{PB}, \end{align*}
where $\omega_{2}=\left( {T^{\sigma_{1}(\beta-1)+\sigma_{1}\gamma+1}}\mathbb{B}(\sigma_{1}\gamma+1,\sigma_{1}(\beta-1)+1) \right)^{\frac{1}{\sigma_{1}}}$.
Thus
\begin{equation}\label{e3.25}
\left| \mathcal{I}^{\beta}_{\tau_{i},\tau}h_{i}(\tau,x(\tau))-\mathcal{I}^{\beta}_{\tau_{i},\tau}h_{i}(\tau,y(\tau))\right|e^{-\theta \tau}\leq\dfrac{L_{h_{i}}\omega_{2}}{\Gamma(\beta)(\theta \sigma_{1}^{*})^{\frac{1}{\sigma_{1}^{*}}}}\|x-y\|_{PB},\;\;\tau\in (\tau_{i},\sigma_{i}]
\end{equation}
By definition of an operator $T$ and using inequations \eqref{e3.19} and \eqref{e3.25}, for any $x,y\in PC(J,\mathbb{R})$, we have:\\
\noindent  Case~(i): For $\tau\in (\sigma_{i},\tau_{i+1}],\; i=1,2,\cdots m,$
\begin{align*}
|& T  x(\tau)-Ty(\tau)|e^{-\theta \tau}\nonumber\\
&\leq    
 \mathcal{I}^{\beta}_{\tau_{i},\sigma_{i}}|h_{i}(\sigma_{i},x(\sigma_{i}))-h_{i}(\sigma_{i},y(\sigma_{i}))|e^{-\theta \tau}\nonumber\\
&\quad+\mathcal{I}^{\alpha}_{\sigma_{i},\tau}\left|f\left(\tau,x(\tau),\int_{\sigma_{i}}^{\tau}h(\sigma,x(\sigma))d\sigma \right)-f\left(\tau,y(\tau),\int_{\sigma_{i}}^{\tau}h(\sigma,y(\sigma))d\sigma\right)\right|e^{-\theta \tau}\nonumber\\
&\leq \left(\dfrac{L_{h_{i}}\omega_{2}}{\Gamma(\beta)(\theta \sigma_{1}^{*})^{\frac{1}{\sigma_{1}^{*}}}}+ \omega_{1}\frac{M_{f}}{\Gamma(\alpha)(\theta \sigma^{*})^{\frac{1}{\sigma^{*}}}} +\frac{N_{f}K_{h}}{{\Gamma(\alpha+1)}}\frac{(\tau_{i+1}-\sigma_{i})^{\alpha+\frac{1}{2}}}{\left( 2\alpha+1\right)^\frac{1}{2}(2\theta)^\frac{1}{2}}\right) \|x-y\|_{PB}
\end{align*}
\noindent  Case~(ii): For $\tau\in (0,\tau_{1}]$, 
\begin{align*}
&| T  x(\tau)-Ty(\tau)|e^{-\theta \tau}\\
&\leq    
 \mathcal{I}^{\alpha}_{0,\tau}\left|f\left(\tau,x(\tau),\int_{0}^{\tau}h(\sigma,x(\sigma))d\sigma\right)-f\left(\tau,y(\tau),\int_{0}^{\tau}h(\sigma,y(\sigma))d\sigma\right)\right|e^{-\theta \tau}\nonumber\\
 &\leq
\left( \omega_{1}\frac{M_{f}}{\Gamma(\alpha)(\theta \sigma^{*})^{\frac{1}{\sigma^{*}}}} +\frac{N_{f}K_{h}}{{\Gamma(\alpha+1)}}\frac{\tau_{1}^{\alpha+\frac{1}{2}}}{\left( 2\alpha+1\right)^\frac{1}{2}(2\theta)^\frac{1}{2}}\right) \|x-y\|_{PB}.
 \end{align*}
 \noindent  Case~(iii): For  $\tau\in(\sigma_{i},\tau_{i}],\; i=1,2,\cdots, m$, 
 \begin{align*}
 | T  x(\tau)-Ty(\tau)|e^{-\theta \tau}&\leq\left| \mathcal{I}^{\beta}_{\tau_{i},\tau}h_{i}(\tau,x(\tau))-\mathcal{I}^{\beta}_{\tau_{i},\tau}h_{i}(\tau,y(\tau))\right|e^{-\theta \tau}\\
 &\leq\dfrac{L_{h_{i}}\omega_{2}}{\Gamma(\beta)(\theta \sigma_{1}^{*})^{\frac{1}{\sigma_{1}^{*}}}}\|x-y\|_{PB}
 \end{align*}
 By cases (i)-(iii), we have
 $$\|Tx-Ty\|_{PB}\leq L_{1} \|x-y\|_{PB}, ~ \mbox{for any} ~x,y\in PC(J,\mathbb{R}), $$
 where
 \begin {align*}
L_{1}=& \max\left\lbrace \dfrac{L_{h_{i}}\omega_{2}}{\Gamma(\beta)(\theta \sigma_{1}^{*})^{\frac{1}{\sigma_{1}^{*}}}}+ \omega_{1}\frac{M_{f}}{\Gamma(\alpha)(\theta \sigma^{*})^{\frac{1}{\sigma^{*}}}}  
\right.\\ &\qquad\qquad\left.
\qquad+\frac{N_{f}K_{h}}{{\Gamma(\alpha+1)}}\frac{(\tau_{i+1}-\sigma_{i})^{\alpha+\frac{1}{2}}}{\left( 2\alpha+1\right)^\frac{1}{2}(2\theta)^\frac{1}{2}}:\;\;  i=0,1,\cdots m\right\rbrace
\end{align*}
By choosing sufficient large value of $\theta$, we get $L_{1} <1$ and in this case $T$ is a contraction and hence $T$ has a unique fixed point, which is the unique solution of \eqref{e3.1}.
\end{proof}

\section{Bielecki-Ulam-Hyers Stability }
We adopt the idea of Wang and Zang \cite {wang1} to investigate the  concepts of Bielecki-Ulam’s type stability 
for the class of nonlinear fractional  order Volterra integrodifferential equation \eqref{e1.1}.

For any $\theta>0,\;\epsilon>0,\;\psi\geq 0,\;\varphi\in PC(J,\mathbb{R_{+}})$ is nondecreasing and $\alpha,\beta \in (0,1),$ consider the following inequalities.
\begin{equation}\label{e2.1}
\begin{cases}
&\left| ^c_{\sigma_{i}}\mathcal{D}_{\tau}^\alpha y(\tau)-f(\tau,y(\tau),\int_{\sigma_{i}}^{\tau}h(\sigma,y(\sigma))d\sigma)\right|\leq\epsilon,~ \tau\in(\sigma_{i},\tau_{i+1}],\; i=0,1,\cdots ,m,\\
& |y(\tau)-\mathcal{I}^{\beta}_{\tau_{i},\tau}h_{i}(\tau,y(\tau))|\leq\epsilon,~ \tau\in(\tau_{i},\sigma_{i}],\; i=1,2,\cdots ,m.  \\
\end{cases}
\end{equation}
and
\begin{equation}\label{e2.2}
\begin{cases}
&\left| ^c_{\sigma_{i}}\mathcal{D}_{\tau}^\alpha y(\tau)-f(\tau,y(\tau),\int_{\sigma_{i}}^{\tau}h(\sigma,y(\sigma))d\sigma)\right|\leq\varphi(\tau),~ \tau\in(\sigma_{i},\tau_{i+1}],\; i=0,1,\cdots ,m,\\
 & |y(\tau)-\mathcal{I}^{\beta}_{\tau_{i},\tau}h_{i}(\tau,y(\tau))|\leq\psi,~ \tau\in(\tau_{i},\sigma_{i}],\; i=1,2,\cdots ,m. 
\end{cases} 
\end{equation}
and
\begin{equation}\label{e2.3}
\begin{cases}
&\left| ^c_{\sigma_{i}}\mathcal{D}_{\tau}^\alpha y(\tau)-f(\tau,y(\tau),\int_{\sigma_{i}}^{\tau}h(\sigma,y(\sigma))d\sigma)\right|\leq\epsilon\varphi(\tau),~ \tau\in(\sigma_{i},\tau_{i+1}],\; i=0,1,\cdots ,m,\\
 & |y(\tau)-\mathcal{I}^{\beta}_{\tau_{i},\tau}h_{i}(\tau,y(\tau))|\leq\epsilon\psi,~ \tau\in(\tau_{i},\sigma_{i}],\; i=1,2,\cdots ,m. 
\end{cases} 
\end{equation}

\begin{definition}\label{def4.1}
The equation \eqref{e1.1} is Bielecki-Ulam-Hyers stable if there exists a real number $ C_{f,\theta,h,\alpha,\beta,h_{i}}>0 $ such that for each $\epsilon>0$ for each solution $y\in PC^{1}(J,\mathbb{R})$ of inequality \eqref{e2.1} there exists a solution $x\in PC^{1}(J,\mathbb{R})$ of equation \eqref{e1.1} with $$ |y(\tau)-x(\tau)|e^{-\theta \tau}\leq C_{f,\theta,h,\alpha,\beta,h_{i}} \epsilon,\;\; \tau\in J.$$
\end{definition}
\begin{definition}\label{def4.2}
The equation \eqref{e1.1} is generalized  Bielecki-Ulam-Hyers stable if there exists $ \mathbb{\theta}_{f,\theta,h,\alpha,\beta,h_{i}}\in C(\mathbb{R_{+}},\mathbb{R_{+}}),~\mathbb{\theta}_{f,\theta,h,\alpha,\beta,h_{i}}(0)=0$ such that for each $\epsilon>0$ for each solution $y\in PC'(J,\mathbb{R})$ of inequality \eqref{e2.1} there exists a solution $x\in PC^{1}(J,\mathbb{R})$ of equation \eqref{e1.1} with
$$ |y(\tau)-x(\tau)|e^{-\theta \tau}\leq \mathbb{\theta}_{f,\theta,h,\alpha,\beta,h_{i}}(\epsilon),\;\; \tau\in J.$$
\end{definition}
\begin{definition}\label{def4.3}
The equation \eqref{e1.1} is Bielecki-Ulam-Hyers-Rassias  stable with respect to $(\phi, \psi)$ if there exists $ C_{f,\theta,h,\alpha,\beta,h_{i},\varphi}>0$ such that for each $\epsilon>0$ for each solution $y\in PC^{1}(J,\mathbb{R})$ of inequality \eqref{e2.3}
 there exists a solution $x\in PC^{1}(J,\mathbb{R})$ of equation \eqref{e1.1} with
$$ |y(\tau)-x(\tau)|e^{-\theta \tau}\leq C_{f,\theta,h,\alpha,\beta,h_{i},\varphi}\epsilon(\psi+\varphi(\tau)) ,\;\; \tau\in J.$$
\end{definition}
\begin{definition}\label{def4.4}
The equation \eqref{e1.1} is generalized  Bielecki-Ulam-Hyers-Rassias  stable with respect to $(\varphi, \psi)$ if there exists $ C_{f,\theta,h,\alpha,\beta,h_{i},\varphi}>0$ such that for each solution $y\in PC^{1}(J,\mathbb{R})$ of inequality \eqref{e2.2} there exists a solution $x\in PC^{1}(J,\mathbb{R})$ of equation \eqref{e1.1} with
$$ |y(\tau)-x(\tau)|e^{-\theta \tau}\leq C_{f,\theta,h,\alpha,\beta,h_{i},\varphi}(\psi+\varphi(\tau)) ,\;\; t\in J.$$
\end{definition}
 \begin{lemma}\label{lm4.5}
If $y\in PC^{1}(J,\mathbb{R})$ is a solution of inequality \eqref{e2.3} then $y$  satisfies the following integral inequalities
 \begin{equation}\label{le4.1}
 \begin{cases}
  & |y(\tau)-\mathcal{I}^{\beta}_{\tau_{i},\tau}h_{i}(\tau,y(\tau))|\leq \epsilon\psi,~ \tau\in(\tau_{i},\sigma_{i}],\; i=1,2,\cdots ,m,  \\
  & \left|y(\tau)-  y(0)-\mathcal{I}^{\alpha}_{0,\tau} f\left(\tau,y(\tau),\int_{0}^{\tau}h(\sigma,y(\sigma)\right)d\sigma\right|\leq \epsilon \mathcal{I}^{\alpha}_{0,\tau}\varphi(\tau), \;\text{if  $\tau\in 
  (0,\tau_{1}],$}\\
 &\left|y(\tau)-\mathcal{I}^{\beta}_{\tau_{i},\sigma_{i}} h_{i}(\sigma_{i},y(\sigma_{i}))-\mathcal{I}^{\alpha}_{\sigma_{i},\tau} f\left(\tau,y(\tau),\int_{\sigma_{i}}^{\tau}h(\sigma,y(\sigma)\right)d\sigma\right|\leq \epsilon\left( \psi+\mathcal{I}^{\alpha}_{\sigma_{i},\tau}\varphi(\tau)\right) \\
&  \tau\in (\sigma_{i},\tau_{i+1}],\; i=1,2,\cdots m.
 \end{cases}
 \end{equation} 
  \end{lemma}
\begin{proof}
If $y\in PC^{1}(J,\mathbb{R})$ is a solution of the inequality \eqref{e2.3} then there exists $H\in PC(J,\mathbb{R})$ and constants $H_{i},\;i=1,2,\cdots m$ (which depend on $y$) such that
\begin{itemize}
\item[(i)] $| H(\tau)|\leq \epsilon\,\varphi(\tau),~ \tau\in J \;\; \text{and} \;\; |H_{i}|\leq  \epsilon\,\psi \;\;\text{for}\;\; i=1,2,\cdots m.$
\item[(ii)] 
$
\begin{cases}
^c_{\sigma_{i}}\mathcal{D}_{\tau}^\alpha y(\tau) =f\left(\tau,y(\tau),\int_{\sigma_{i}}^{\tau}h(\sigma,y(\sigma))d\sigma\right)+ H(\tau),\, \tau\in(\sigma_{i},\tau_{i+1}],\; i=0,1,\cdots ,m, \\
y(\tau) =\mathcal{I}^{\beta}_{\tau_{i},\tau}h_{i}(\tau,y(\tau))+ H_{i},\;\tau\in (\tau_{i},\sigma_{i}],\;i=1,2,\cdots m,\\
\end{cases}
$
\end{itemize}
In the view of Lemma \ref{lm3.1}, above equation is equivalent to the integral equations
\begin{equation*}
y(\tau) =
\begin{cases}
\mathcal{I}^{\beta}_{\tau_{i},\tau}h_{i}(\tau,y(\tau))+\epsilon H_{i}, \; \tau\in(\tau_{i},\sigma_{i}],\; i=1,\cdots ,m,   \\
y(0)+\mathcal{I}^{\alpha}_{0,\tau}\left[ f\left(\tau,y(\tau),\int_{0}^{\tau}h(\sigma,y(\sigma))d\sigma\right)+\epsilon H(\tau)\right],  \tau\in [0,\tau_{1}],\\
\mathcal{I}^{\beta}_{\tau_{i},\sigma_{i}} h_{i}(\sigma_{i},y(\sigma_{i}))+\epsilon H_{i} +\mathcal{I}^{\alpha}_{\sigma_{i},\tau}\left[ f\left(\tau,y(\tau),\int_{\sigma_{i}}^{\tau}h(\sigma,y(\sigma))d\sigma\right)+\epsilon H(\tau)\right],\\
  \tau\in (\sigma_{i},\tau_{i+1}],i=1,\cdots m.
\end{cases}
\end{equation*}
For any $\tau\in (\sigma_{i},\tau_{i+1}],\; i=1,2,\cdots m,$ 
\begin{align*}
\left| y(\tau)-  \mathcal{I}^{\beta}_{\tau_{i},\sigma_{i}} h_{i}(\sigma_{i},y(\sigma_{i}))-  \mathcal{I}^{\alpha}_{\sigma_{i},\tau} f\left(\tau,y(\tau),\int_{\sigma_{i}}^{\tau}h(\sigma,y(\sigma)\right)d\sigma\right|
& =| H_{i}+ \mathcal{I}^{\alpha}_{\sigma_{i},\tau}H(\tau)|\\
&\leq | H_{i}|+ \mathcal{I}^{\alpha}_{\sigma_{i},\tau}|H(\tau)|\\
& \leq \epsilon\left( \psi+\mathcal{I}^{\alpha}_{\sigma_{i},\tau}\varphi(\tau)\right). 
\end{align*}
For $\tau\in (0,\tau_{1}]$,
\begin{align*}
\left|y(\tau)- y(0)-\mathcal{I}^{\alpha}_{0,\tau} f\left(\tau,y(\tau),\int_{0}^{\tau}h(\sigma,y(\sigma)\right)d\sigma\right|
=\mathcal{I}^{\alpha}_{0,\tau}| H(\tau)|\leq \epsilon \mathcal{I}^{\alpha}_{0,\tau}\varphi(\tau).
\end{align*}
For $\tau\in (\tau_{i},\sigma_{i}],\; i=1,2,\cdots m,$
$$\left|y(\tau)-\mathcal{I}^{\beta}_{\tau_{i},\sigma_{i}} h_{i}(\tau,y(\tau))\right|
\leq | H_{i}|\leq \epsilon \psi,\; \tau\in (\tau_{i},\sigma_{i}],\; i=1,2,\cdots m.$$ 
The last three inequalities are the required equivalent integral inequalities in \eqref{le4.1}.
\end{proof}
Following additional assumption is needed to prove the Bielecki-Ulam-Hyers-Rassias stability of equation \eqref{e1.1}.
\begin{itemize}
\item[(H4)] Let $\varphi\in C(J,\mathbb{R}_{+})$  is nondecreasing and there  exists $c_{\varphi>0}$ such that 
$$\mathcal{I}^{\alpha}_{0,\tau} \varphi\leq c_\varphi \varphi(\tau),\; \tau\in J.$$  \end{itemize}
\begin{theorem}\label{Thm4.5}
Assume that hypotheses $(\tilde{H}1)$, $(H2)$, $(\tilde{H}3)$ and $(H4)$ holds. Then, the equation \eqref{e1.1} is  Bielecki-Ulam-Hyers-Rassias stable with respect to $(\varphi, \psi)$, where $\alpha,~\beta \in (0,1)$.
\end{theorem}
\begin{proof}
Let $y\in PC^{1}(J,\mathbb{R})$ be a solution of inequality \eqref{e2.3}. Then by lemma \ref{lm4.5}, $y$ satisfies the integral inequalities \begin{equation}\label{e4.5}
 \begin{cases}
   |y(\tau)-\mathcal{I}^{\beta}_{\tau_{i},\tau}h_{i}(\tau,y(\tau))|\leq \epsilon\psi,~ \tau\in(\tau_{i},\sigma_{i}],\; i=1,2,\cdots ,m,  \\
  \, \left|y(\tau)-  y(0)-\mathcal{I}^{\alpha}_{0,\tau} f\left(\tau,y(\tau),\int_{0}^{\tau}h(\sigma,y(\sigma)\right)d\sigma\right|\leq \epsilon c_{\varphi}\varphi(\tau), \;  \tau\in (0,\tau_{1}],\\
 \left|y(\tau)-\mathcal{I}^{\beta}_{\tau_{i},\sigma_{i}} h_{i}(\sigma_{i},y(\sigma_{i}))-\mathcal{I}^{\alpha}_{\sigma_{i},\tau} f\left(\tau,y(\tau),\int_{\sigma_{i}}^{\tau}h(\sigma,y(\sigma)\right)d\sigma\right|\leq \epsilon\left( \psi+c_{\varphi}\varphi(\tau)\right) ,\\
  \tau\in (\sigma_{i},\tau_{i+1}],\; i=1,2,\cdots m.
 \end{cases}
 \end{equation} 
 Denote by $x$ the classical solution of  fractional Volterra integrodifferential equations 
\begin{equation}\label{a4.5}
\begin{cases}
 ^c_{\sigma_{i}}\mathcal{D}_{\tau}^\alpha x(\tau)=f\left(\tau,x(\tau),\int_{\sigma_{i}}^{\tau}h(\sigma,x(\sigma))d\sigma\right),~ \tau\in(\sigma_{i},\tau_{i+1}],\; i=0,1,\cdots ,m,\\
  x(\tau)=\mathcal{I}^{\beta}_{\tau_{i},\tau}h_{i}(\tau,x(\tau)),~ \tau\in(\tau_{i},\sigma_{i}],\; i=1,2,\cdots ,m, \\
  x(0)=y(0)
\end{cases} 
\end{equation}
Then, in the view of lemma \ref{lm3.1}, $x$ satisfies the fractional integral equations
\begin{equation}\label{e14.6}
x(\tau) =
\begin{cases}
y(0),\;  \tau=0,\\
\mathcal{I}^{\beta}_{\tau_{i},\tau}h_{i}(\tau,x(\tau)),\; \tau\in(\tau_{i},\sigma_{i}],\; i=1,\cdots ,m , \\
y(0)+\mathcal{I}^{\alpha}_{0,\tau}f\left(\tau,x(\tau),\int_{0}^{\tau}h(\sigma,x(\sigma))d\sigma\right),\;\tau\in (0,\tau_{1}],\\
\mathcal{I}^{\beta}_{\tau_{i},\sigma_{i}}h_{i}(\sigma_{i},x(\sigma_{i}))+\mathcal{I}^{\alpha}_{\sigma_{i},\tau}f\left(\tau,x(\tau),\int_{\sigma_{i}}^{\tau}h(\sigma,x(\sigma))d\sigma\right),
\tau\in (\sigma_{i},\tau_{i+1}],i=1,\cdots m
\end{cases} 
\end{equation}
Proceeding as in the proof of Theorem \ref{Thm2}, for any $\tau\in (\sigma_{i},\tau_{i+1}],\; i=1,2,\cdots m$, we obtain 
\begin {align*}
&|y(\tau)-x(\tau)|e^{-\theta \tau}=\left| y(\tau)-\left( \mathcal{I}^{\beta}_{\tau_{i},\sigma_{i}}h_{i}(\sigma_{i},x(\sigma_{i}))+\mathcal{I}^{\alpha}_{\sigma_{i},\tau}f\left( \tau,x(\tau),\int_{\sigma_{i}}^{\tau}h(\sigma,x(\sigma))d\sigma\right) \right) \right| e^{-\theta \tau}\\
&\leq \left| y(\tau)- \mathcal{I}^{\beta}_{\tau_{i},\sigma_{i}}h_{i}(\sigma_{i},y(\sigma_{i}))-\mathcal{I}^{\alpha}_{\sigma_{i},t}f\left( \tau,y(\tau),\int_{\sigma_{i}}^{\tau}h(\sigma,y(\sigma))d\sigma\right)  \right|e^{-\theta \tau}\\
&\qquad+|\mathcal{I}^{\beta}_{\tau_{i},\sigma_{i}}h_{i}(\sigma_{i},y(\sigma_{i}))-\mathcal{I}^{\beta}_{\tau_{i},\sigma_{i}}h_{i}(\sigma_{i},x(\sigma_{i}))|e^{-\theta \tau}\\
&\qquad+\left| \mathcal{I}^{\alpha}_{\sigma_{i},\tau}f\left(\tau,y(\tau),\int_{\sigma_{i}}^{\tau}h(\sigma,y(\sigma))d\sigma\right)-\mathcal{I}^{\alpha}_{\sigma_{i},\tau}f\left( \tau,x(\tau),\int_{\sigma_{i}}^{\tau}h(\sigma,x(\sigma))d\sigma\right) \right| e^{-\theta \tau}\\
&=\epsilon\left( \psi+c_{\varphi}\varphi(\tau)\right)e^{-\theta \tau}+
\dfrac{L_{h_{i}}\omega_{2}}{\Gamma(\beta)(\theta \sigma_{1}^{*})^{\frac{1}{\sigma_{1}^{*}}}}
\left(\sup_{\tau\in (\tau_{i},\sigma_{i}]}e^{-\theta \tau}|y(\tau)-x(\tau)|\right)\\
&+\left[\frac{\omega_{1}M_{f}}{\Gamma(\alpha)(\theta \sigma^{*})^{\frac{1}{\sigma^{*}}}} +\frac{N_{f}K_{h}(\tau_{i+1}-\sigma_{i})^{\alpha+\frac{1}{2}}}{{\Gamma(\alpha+1)}\left( 2\alpha+1\right)^\frac{1}{2}(2\theta)^\frac{1}{2}}\right] \left(\sup_{\tau\in (\sigma_{i},\tau_{i+1}]}e^{-\theta \tau}|y(\tau)-x(\tau)|\right)
\end{align*}
This gives 
\begin{align*}
  &\sup_{\tau\in (\sigma_{i},\tau_{i+1}]}e^{-\theta \tau}|x(\tau)-y(\tau)|\\
  &\qquad\leq\epsilon (1+c_{\varphi})\left(  \psi +\varphi(\tau)\right)e^{-\theta \sigma_{i}} 
+\left(\dfrac{L_{h_{i}}\omega_{2}}{\Gamma(\beta)(\theta \sigma_{1}^{*})^{\frac{1}{\sigma_{1}^{*}}}}
+ \frac{\omega_{1}M_{f}}{\Gamma(\alpha)(\theta \sigma^{*})^{\frac{1}{\sigma^{*}}}}\right.\\
&\left.\qquad +\frac{N_{f}K_{h}(\tau_{i+1}-\sigma_{i})^{\alpha+\frac{1}{2}}}{{\Gamma(\alpha+1)}\left( 2\alpha+1\right)^\frac{1}{2}(2\theta)^\frac{1}{2}}\right) \sup_{\tau\in (\sigma_{i},\tau_{i+1}]}e^{-\theta \tau}|x(\tau)-y(\tau)|
\end{align*}
Therefore,
\begin{align}\label{nal4.10}
&\sup_{\tau\in (\sigma_{i},\tau_{i+1}]}e^{-\theta \tau}|x(\tau)-y(\tau)|\nonumber\\
&\leq\dfrac{\epsilon (1+c_{\varphi})\left(  \psi +\varphi(\tau)\right)e^{-\theta \sigma_{i}}}{ 1-\left(\dfrac{L_{h_{i}}\omega_{2}}{\Gamma(\beta)(\theta \sigma_{1}^{*})^{\frac{1}{\sigma_{1}^{*}}}}
 + \dfrac{\omega_{1}M_{f}}{\Gamma(\alpha)(\theta \sigma^{*})^{\frac{1}{\sigma^{*}}}}+\dfrac{N_{f}K_{h}(\tau_{i+1}-\sigma_{i})^{\alpha+\frac{1}{2}}}{{\Gamma(\alpha+1)}\left( 2\alpha+1\right)^\frac{1}{2}(2\theta)^\frac{1}{2}}\right)}
\end{align}
Now, for any $\tau\in (\tau_{i},\sigma_{i}], i=1,2,\cdots m$,
\begin{align*}
|y(\tau)-x(\tau)|e^{-\theta \tau}&\leq |y(\tau)-\mathcal{I}^{\beta}_{\tau_{i},\tau}h_{i}(\tau,x(\tau))|e^{-\theta \tau}\nonumber\\
&\leq |y(\tau)-\mathcal{I}^{\beta}_{\tau_{i},\tau}h_{i}(\tau,y(\tau))|e^{-\theta \tau}+|\mathcal{I}^{\beta}_{\tau_{i},\tau}h_{i}(\tau,y(\tau))-\mathcal{I}^{\beta}_{\tau_{i},\tau}h_{i}(\tau,x(\tau))|e^{-\theta \tau}\nonumber\\
&\leq\epsilon \psi e^{-\theta \tau_{i}}+\dfrac{L_{h_{i}}\omega_{2}}{\Gamma(\beta)(\theta \sigma_{1}^{*})^{\frac{1}{\sigma_{1}^{*}}}}\left( \sup_{\tau\in (\tau_{i},\sigma_{i}]}e^{-\theta \tau}|x(\tau)-y(\tau)|\right) 
\end{align*}
Therefore 
\begin{align*}
\sup_{\tau\in (\tau_{i},\sigma_{i}]}e^{-\theta \tau}|x(\tau)-y(\tau)|
&\leq\epsilon \psi e^{-\theta \tau_{i}}+\dfrac{L_{h_{i}}\omega_{2}}{\Gamma(\beta)(\theta \sigma_{1}^{*})^{\frac{1}{\sigma_{1}^{*}}}}\left( \sup_{\tau\in (\tau_{i},\sigma_{i}]}e^{-\theta \tau}|x(\tau)-y(\tau)|\right)
\end{align*}
Thus we have
\begin{equation}\label{ne4.11}
 \sup_{\tau\in (\tau_{i},\sigma_{i}]}e^{-\theta \tau}|x(\tau)-y(\tau)|
 \leq \dfrac{\epsilon\psi}{\left( 1-\dfrac{L_{h_{i}}\omega_{2}}{\Gamma(\beta)(\theta \sigma_{1}^{*})^{\frac{1}{\sigma_{1}^{*}}}}\right) } e^{-\theta \tau_{i}}
\end{equation}
Next, for any $\tau\in (0,\tau_{1}]$,
\begin{align*}
&|y(\tau)-x(\tau)|e^{-\theta \tau}\\
&\leq \left|y(\tau)-y(0)-\mathcal{I}^{\alpha}_{0,\tau}f\left( \tau,x(\tau),\int_{0}^{\tau}h(\sigma,x(\sigma))d\sigma\right) \right|e^{-\theta \tau}\nonumber\\
&\leq \left|y(\tau)-y(0)-\mathcal{I}^{\alpha}_{0,\tau}f\left( \tau,y(\tau),\int_{0}^{\tau}h(\sigma,y(\sigma))d\sigma\right) \right|e^{-\theta \tau}\nonumber\\
&\qquad+\mathcal{I}^{\alpha}_{0,\tau}\left|f(\tau,y(\tau),\int_{0}^{\tau}h(\sigma,y(\sigma))d\sigma)-f(\tau,x(\tau),\int_{0}^{\tau}h(\sigma,x(\sigma))d\sigma)\right|e^{-\theta \tau}\nonumber\\
& \leq \epsilon c_{\varphi}\varphi(t)+\left( \frac{\omega_{1}M_{f}}{\Gamma(\alpha)(\theta \sigma^{*})^{\frac{1}{\sigma^{*}}}} +\frac{N_{f}K_{h}}{{\Gamma(\alpha+1)}}\frac{\tau_{1}^{\alpha+\frac{1}{2}}}{\left( 2\alpha+1\right)^\frac{1}{2}(2\theta)^\frac{1}{2}}\right) \sup_{\tau\in (0,\tau_{1}]} e^{-\theta \tau}|x(\tau)-y(\tau)|
\end{align*}
Just computed as above, we get
\begin{align}\label{e4.9}
 \sup_{\tau\in (0,\tau_{1}]} |y(\tau)-x(\tau)|e^{-\theta \tau}
 \leq \frac{\epsilon c_{\varphi}\varphi(\tau) }{ 1-\left( \dfrac{\omega_{1}M_{f}}{\Gamma(\alpha)(\theta \sigma^{*})^{\frac{1}{\sigma^{*}}}} +\dfrac{N_{f}K_{h}}{{\Gamma(\alpha+1)}}\dfrac{\tau_{1}^{\alpha+\frac{1}{2}}}{\left( 2\alpha+1\right)^\frac{1}{2}(2\theta)^\frac{1}{2}}\right)} \nonumber\\
\end{align}
Since $$ J=(0,T]=\left[ \bigcup_{i=0}^{m}(\sigma_{i},\tau_{i+1}]\right] \left[ \bigcup_{i=1}^{m}(\tau_{i},\sigma_{i}]\right], $$
from the  inequalities \eqref{nal4.10},\eqref{ne4.11} and \eqref{e4.9}, we obtain 
\begin{align}\label{a4.13}
 \sup_{\tau\in (0,T]} |y(\tau)-x(\tau)|e^{-\theta \tau}
& \leq \sup_{\tau\in (0,\tau_{1}]} |y(\tau)-x(\tau)|e^{-\theta \tau}
+\sum_{i=1}^{m}  \sup_{\tau\in (\sigma_{i},\tau_{i+1}]}|y(\tau)-x(\tau)|e^{-\theta \tau}\nonumber\\
&\qquad+ \sum_{i=1}^{m}\sup_{\tau\in (\tau_{i},\sigma_{i}]} |y(\tau)-x(\tau)|e^{-\theta \tau}\nonumber \\
& \quad\leq \epsilon c_{f,\theta,h,\alpha,\beta,h_{i},\varphi}(\psi+\varphi(\tau))
\end{align}
where
\begin{align*}
c_{f,\theta,h,\alpha,\beta,h_{i},\varphi}&=\frac{ c_{\varphi} }{1-\left( \dfrac{\omega_{1}M_{f}}{\Gamma(\alpha)(\theta \sigma^{*})^{\frac{1}{\sigma^{*}}}} +\dfrac{N_{f}K_{h}}{{\Gamma(\alpha+1)}}\dfrac{\tau_{1}^{\alpha+\frac{1}{2}}}{\left( 2\alpha+1\right)^\frac{1}{2}(2\theta)^\frac{1}{2}}\right)  }+\large\sum_{i=1}^{m} \dfrac{\psi}{\left( 1-\dfrac{L_{h_{i}}\omega_{2}}{\Gamma(\beta)(\theta \sigma_{1}^{*})^{\frac{1}{\sigma_{1}^{*}}}}\right)} \\ 
&\qquad+\sum_{i=1}^{m} \dfrac{1+c_{\varphi}}{1-\left(\dfrac{L_{h_{i}}\omega_{2}}{\Gamma(\beta)(\theta \sigma_{1}^{*})^{\frac{1}{\sigma_{1}^{*}}}}
 + \dfrac{\omega_{1}M_{f}}{\Gamma(\alpha)(\theta \sigma^{*})^{\frac{1}{\sigma^{*}}}}+\dfrac{N_{f}K_{h}(\tau_{i+1}-\sigma_{i})^{\alpha+\frac{1}{2}}}{{\Gamma(\alpha+1)}\left( 2\alpha+1\right)^\frac{1}{2}(2\theta)^\frac{1}{2}}\right)}
\end{align*}
Finally, from inequality \eqref{a4.13}, we have
\begin{equation}\label{e4.14}
|y(\tau)-x(\tau)|e^{-\theta \tau}\leq \epsilon c_{f,\theta,h,\alpha,\beta,h_{i},\varphi}(\psi+\varphi(\tau)),\; \tau\in J
\end{equation}
This shows that equation \eqref{e1.1}  Bielecki-Ulam-Hyers-Rassias  stable with respect to $(\varphi, \psi)$.
\end{proof}
\begin{corollary} \label{c4.7}
Assume that hypotheses $(\tilde{H}1)$, $(H2)$, $(\tilde{H}3)$ and $(H4)$ holds. Then, the equation \eqref{e1.1} is  generalized Bielecki-Ulam-Hyers-Rassias stable with respect to $(\varphi, \psi)$, where $\alpha, \beta \in (0,1)$.
\end{corollary}
\begin{proof}
Set $\epsilon=1$ in the proof of Theorem 5.3, we obtain
\begin{align}
|y(\tau)-x(\tau)|e^{-\theta \tau}\leq c_{f,\theta,h,\alpha,\beta,{i},\varphi}(\psi+\varphi(\tau)),\; \tau\in J.
\end{align}
This proves the equation \eqref{e1.1} is  generalized Bielecki-Ulam-Hyers-Rassias stable with respect to $(\varphi, \psi)$.
\end{proof}
\begin{corollary}\label{c4.8}
Assume that hypotheses $(\tilde{H}1)$, $(H2)$, $(\tilde{H}3)$ and $(H4)$ holds. Then, the equation \eqref{e1.1} is   Bielecki-Ulam-Hyers stable, where $\alpha, \beta \in (0,1)$.
\end{corollary}
\begin{proof}
Set $\psi=1$ and $\varphi(\tau)=1,\; \tau\in J$ in theorem proof of Theorem \ref{Thm4.5}. Then $\varphi \in C(J,\mathbb{R}_+)$ and 
$\mathcal{I}^{\alpha}_{0,\tau} \varphi\leq c_\varphi \varphi(\tau),\; \tau\in J,$
where $c_\varphi=\dfrac{T^{\alpha}}{\Gamma(1+\alpha)}$. Thus the hypothesis $(H4)$ is satisfied. Further, we  have
\begin{equation} \label{BUH}
|y(\tau)-x(\tau)|e^{-\theta \tau}\leq \epsilon c_{f,\theta,h,\alpha,\beta,h_{i}},\; \tau\in J.
\end{equation}
Therefore the equation \eqref{e1.1} is   Bielecki-Ulam-Hyers stable.
\end{proof}
\begin{corollary}
Assume that hypotheses $(\tilde{H}1)$, $(H2)$, $(\tilde{H}3)$ and $(H4)$ holds. Then, the equation \eqref{e1.1} is  generalized Bielecki-Ulam-Hyers stable, where $\alpha, \beta \in (0,1)$.
\end{corollary}
\begin{proof}
Define $\mathbb{\theta}_{f,\theta,h,\alpha,\beta,h_{i}}:\mathbb{R}_{+}\to \mathbb{R}_{+}$  by $\mathbb{\theta}_{f,\theta,h,\alpha,\beta,h_{i}}(\epsilon)=\epsilon c_{f,\theta,h,\alpha,\beta,h_{i}}.$
Then $ \mathbb{\theta}_{f,\theta,h,\alpha,\beta,h_{i}}\in C(\mathbb{R_{+}},\mathbb{R_{+}})$ and $\mathbb{\theta}_{f,\theta,h,\alpha,\beta,h_{i}}(0)=0.$
Further from the inequation \eqref{BUH}, we have
$$|y(\tau)-x(\tau)|e^{-\theta t}\leq \theta(\epsilon),\; \tau\in J,$$
which shows that \eqref{e1.1} is  generalized Bielecki-Ulam-Hyers stable. 
\end{proof}
\begin{remark}
Under the assumptions of Theorem \ref{Thm1} one can obtain the Bielecki-Ulam-Hyers stability and Bielecki-Ulam-Hyers-Rassias stability of equation when $\alpha, \beta \in (\frac{1}{2},1).$
\end{remark}
 
\section{Example}
\noindent \textbf{Example 5.1:}
Consider Caputo fractional differential equation with fractional integrable impulse
\begin{align}\label{ex1}
^{c}_{\sigma_{i}}\mathcal{D}_{\tau}^{\frac{2}{3}}x(\tau)&= f\left(\tau,x(\tau),\int_{\sigma_{i}}^{\tau}h( \sigma,x(\sigma))d\sigma\right), \tau\in (\sigma_{i},\tau_{i+1}],i=0,1\\
x(\tau)&= \mathcal{I}_{\tau_{1},\tau}^{\frac{2}{3}}h_{1}(\tau,x(\tau)), \;\tau\in (\tau_{1},\sigma_{1}], \label{ex2}\\
x(0)&=0 \label{ex3}
\end{align}
where $ 0=\sigma_{0}=\tau_{0}<\tau_{1}=1< \sigma_{1}=2< \tau_{2}=3$  and the  functions $f:[0,3]\times \mathbb{R}\times \mathbb{R}\to\mathbb{R}; ~h:[0,3]\times \mathbb{R}\to\mathbb{R},~h_{1}:(1,2]\times \mathbb{R}\to\mathbb{R}$  are defined as follows:
\begin{align*}
f\left(\tau,x(\tau),\int_{\sigma_{i}}^{\tau}h(\sigma,x(\sigma))d\sigma\right)&=g(\tau)+ \frac{e^{-\tau^{2}}}{4}\left( \sin(x(\tau))+\cos x(\tau)\right)+\int_{\sigma_{i}}^{\tau}\frac{\sigma}{e^{\sigma^{2}}}\sin(x(\sigma))d\sigma\\
h(\tau,x(\tau))&=\frac{\tau}{e^{\tau^{2}}}\sin(x(\tau))\\
h_{1}(\tau,x(\tau))&= \dfrac{1}{\Gamma(\frac{4}{3})} \frac{\tau (\tau-1)^{\frac{1}{3}}}{(\tau-4)}\left( \dfrac{|x(\tau)|-3}{|x(\tau)|+1}\right)
\end{align*}
where 
\begin{equation*}
g(\tau)=
 \begin{cases}
  & \dfrac{9}{2\Gamma(\frac{1}{3})}\tau^{\frac{4}{3}}-\frac{1}{4},\quad \tau\in [0,1]
 \\
 & 0,\;\tau\in (1,2]
 \\
   &\dfrac{9}{2\Gamma(\frac{1}{3})}(\tau-2)^{\frac{4}{3}}-\frac{1}{4e^{4}}\left( \cos 4+\sin 4\right),\; \;  \tau\in (2,3],\\
 \end{cases}
 \end{equation*} 
For any $x,y,\bar{x},\bar{y}\in PC(J,\mathbb{R})$ and $\tau\in [0,3] $,
 \begin{align*}
|& f(\tau,x,y)-f(\tau,\bar{x},\bar {y})|\\
&=\left| \left( g(\tau)+ \frac{e^{-\tau^{2}}}{4}\left( \sin x+\cos x\right)+y\right)  -\left( g(\tau)+\frac{e^{-\tau^{2}}}{4}\left( \sin \bar{x}+\cos \bar{x}\right)+\bar{y}\right) \right| \\
&\leq\frac{e^{-\tau^{2}}}{4}\left(|\sin x-\sin \bar{x}|+|\cos x-\cos \bar{x}| \right)+|y-\bar{y}| \\
& = \frac{e^{-\tau^{2}}}{4}\left(\left|2\cos\left(\frac{x+\bar{x}}{2}\right) \sin\left(\frac{x-\bar{x}}{2}\right)\right|+\left|2\sin\left(\frac{x+\bar{x}}{2}\right) \sin(\frac{x-\bar{x}}{2})\right|  \right)+ |y-\bar{y}|\\
& \leq \frac{e^{-\tau^{2}}}{2} |x-\bar{x}|+|y-\bar{y}| \leq \frac{1}{2}|x-\bar{x}|+|y-\bar{y}| 
 \end{align*}
Further, for  $\tau\in [1,2]$ 
\begin{align*}
|h_{1}(\tau,x)-h_{1}(\tau,\bar{x})|&=\left| \dfrac{1}{\Gamma(\frac{4}{3})} \frac{\tau (\tau-1)^{\frac{1}{3}}}{(\tau-4)}\left( \dfrac{|x|-3}{|x|+1}\right)-\dfrac{1}{\Gamma(\frac{4}{3})} \frac{\tau (\tau-1)^{\frac{1}{3}}}{(\tau-4)}\left( \dfrac{|y|-3}{|y|+1}\right)\right| \\
&\leq \frac{1}{\Gamma(\frac{4}{3})}\left|\dfrac{|x|-3}{|x|+1}-\dfrac{|y|-3}{|y|+1} \right| \leq\frac{4}{\Gamma(\frac{4}{3})} \dfrac{||x|-|y||}{(|x|+1)(|y|+1)} \\
&\leq\frac{4}{\Gamma(\frac{4}{3})}|x-y|
\end{align*}
and  for $\tau\in [0,3] $, we have 
\begin{align*}
|h(\tau,x(\tau))-h(\tau,\bar{x})|&=\left|\frac{\tau}{e^{\tau^{2}}}\sin x-\frac{\tau}{e^{\tau^{2}}}\sin \bar{x}\right|= \frac{\tau}{e^{\tau^{2}}}\left|\sin x- \sin \bar{x}\right|\\
& =\frac{\tau}{e^{\tau^{2}}}\left|\sin x- \sin \bar{x}\right|=\frac{\tau}{e^{\tau^{2}}}\left| 2\cos(\frac{x+\bar{x}}{2})\sin(\frac{x-\bar{x}}{2})\right| \\
 &\leq6\left|\frac{x-\bar{x}}{2}\right|=3|x-\bar{x}|
\end{align*}
Thus $f,h$ and $h_{1}$ satisfies Lipschitz condition with Lipschitz constants
$$M_{f}=\frac{1}{2}, ~N_{f}=1,~K_{h}=3,~L_{h_{1}}=\dfrac{4}{\Gamma(\frac{4}{3})}. $$
Thus by  Theorem \ref{Thm1}, the problem \eqref{ex1}-\eqref{ex3}  has a unique  solution on $[0,3]$.\\
Let $\varphi(\tau)=2.8361\mathbb{E}_{\frac{2}{3}}(\tau^{\frac{2}{3}}),\; \tau\in J$ and $\psi=0$.
Then $\varphi \in C(J,\mathbb{R}_+)$ is nondecreasing and satisfy the condition
\begin{align*}
\mathcal{I}_{0,\tau}^{\frac{2}{3}}\varphi(\tau)=2.8361\mathcal{I}_{\sigma_{i},\tau}^{\frac{2}{3}}(\mathbb{E}_{\frac{2}{3}}(\tau^{\frac{2}{3}}))\leq 2.8361\mathbb{E}_{\frac{2}{3}}(\tau^{\frac{2}{3}})=c_{\varphi}\varphi(\tau), \tau \in J, 
\end{align*}
where $c_{\varphi}=1$. Note that all the assumptions of the Corollary \ref{c4.7} holds. therefore problem 
\eqref{ex1}-\eqref{ex2} is   generalized Bielecki-Ulam-Hyers-Rassias stable with respect to $(\varphi, \psi)$.

Next, we shall discuss the generalized Bielecki-Ulam-Hyers-Rassias stablity of the equation \eqref{ex1}-\eqref{ex2} by showing that there exist an exact solution $x(\tau)$ of the problem \eqref{ex1}-\eqref{ex3} corresponding to  $(\varphi, \psi)$ and the given solution $y(\tau)$ of the inequalities
\begin{align}
\left| ^{c}_{\sigma_{i}}\mathcal{D}_{\tau}^{\frac{2}{3}}y(\tau)- f(\tau,y(\tau),\int_{\sigma_{i}}^{\tau}h(\sigma,y(\sigma))d\sigma)\right|&\leq \varphi(\tau)  , \tau\in (\sigma_{i},\tau_{i+1}],i=0,1\label{a5.4}\\
\left| y(\tau)- \mathcal{I}_{\tau_{1},\tau}^{\frac{2}{3}}h_{1}(\tau,y(\tau))\right| &\leq 0, \;\tau\in (\tau_{1},\sigma_{1}] \label{a5.5}.
\end{align}
where $ 0=\sigma_{0}=\tau_{0}<\tau_{1}=1< \sigma_{1}=2< \tau_{2}=3$.

Let
\begin{equation*}
y(\tau)=
\begin{cases}
& \tau,\; \tau\in [0,1]\cup (2,3]\\
&(\tau-1),\; \tau\in(1,2]
\end{cases}
\end{equation*}
Then for $\tau\in(0,1]$, 
\begin{align*}
&\left| ^{c}_{0}\mathcal{D}_{\tau}^{\frac{2}{3}}y(t)- g(\tau)- \frac{e^{-\tau^{2}}}{4}\left( \sin y(\tau)+\cos y(\tau)\right)-\int_{0}^{\tau}\frac{\sigma}{e^{\sigma^{2}}}\sin y(\sigma)d\sigma \right|\nonumber\\
&=\left| ^{c}_{0}\mathcal{D}_{\tau}^{\frac{2}{3}}(\tau)-\dfrac{9}{2\Gamma(\frac{1}{3})}\tau^{\frac{4}{3}}+\frac{1}{4} - \frac{e^{-\tau^{2}}}{4}\left( \sin \tau+\cos \tau\right)-\int_{0}^{\tau}\frac{\sigma}{e^{\sigma^{2}}}\sin\sigma  d\sigma \right|\nonumber\\
&\leq \dfrac{1}{\Gamma(\frac{4}{3})}\tau^{\frac{1}{3}}+\dfrac{9}{2\Gamma(\frac{1}{3})}\tau^{\frac{4}{3}}+\frac{1}{4} + \frac{e^{-\tau^{2}}}{4}\left( |\sin \tau|+|\cos \tau|\right)+\int_{0}^{\tau}\frac{\sigma}{e^{\sigma^{2}}}|\sin\sigma|  d\sigma \nonumber\\
&\leq 4.5495
\end{align*}  
and for $\tau\in(2,3]$
\begin{align*}
&\left| ^{c}_{2}\mathcal{D}_{\tau}^{\frac{2}{3}}y(\tau)- g(\tau)- \frac{e^{-\tau^{2}}}{4}\left( \sin(y(\tau))+\cos y(\tau)\right)-\int_{2}^{\tau}\frac{\sigma}{e^{\sigma^{2}}}\sin(y(\sigma))d\sigma \right|\nonumber\\
&=\left| ^{c}_{2}\mathcal{D}_{\tau}^{\frac{2}{3}}[(\tau-2)+2]-\dfrac{9}{2\Gamma(\frac{1}{3})}(\tau-2)^{\frac{4}{3}}+\frac{1}{4e^{4}}\left( \cos 4+\sin 4\right)\right.\nonumber\\
&\left.\qquad- \frac{e^{-\tau^{2}}}{4}\left( \sin \tau+\cos \tau\right)-\int_{2}^{\tau}\frac{\sigma}{e^{\sigma^{2}}}\sin \sigma d\sigma \right|\nonumber\\
&\leq  \dfrac{1}{\Gamma(\frac{4}{3})}(\tau-2)^{\frac{1}{3}}+0+\dfrac{9}{2\Gamma(\frac{1}{3})}(\tau-2)^{\frac{4}{3}}\nonumber\\
&\qquad+\frac{1}{4e^{4}}\left( |\cos 4|+|\sin 4|\right)+ \frac{1}{4e^{4}}\left(| \sin \tau|+|\cos \tau|\right)+\int_{2}^{\tau}\frac{\sigma}{e^{\sigma^{2}}}|\sin \sigma| d\sigma \nonumber\\
&\leq 2.8361
\end{align*}
Therefore for  $\tau\in (\sigma_{i},\tau_{i}],\; i=0,1$, we have 
\begin{align*}
&\left| ^{c}_{\sigma_{i}}\mathcal{D}_{\tau}^{\frac{2}{3}}y(\tau)- g(\tau)- \frac{e^{-\tau^{2}}}{4}\left( \sin y(\tau)+\cos y(\tau)\right)-\int_{\sigma_{i}}^{\tau}\frac{\sigma}{e^{\sigma^{2}}}\sin y(\sigma)d\sigma \right|\\
&\qquad\leq \min \left\lbrace  4.5495,2.8361\right\rbrace =2.8361\leq2.8361\mathbb{E}_{\frac{2}{3}}(\tau^{\frac{2}{3}})=\varphi(\tau)
\end{align*}
Also, for $\tau\in (1,2]$, we get 
\begin{align*} |y(\tau)-\mathcal{I}_{1,\tau}^{\frac{2}{3}}(h_{1}(\tau,y(\tau)))|=|(\tau-1)-(\tau-1)|=0
\end{align*} 
Hence $y(\tau)$ is a solution of an inequality \eqref{a5.4}-\eqref{a5.5}.

Next, one can easily verify that 
 \begin{equation*}
x(\tau)=
\begin{cases}
& \tau^{2},\quad \tau\in (0,1]\cup (2,3]
\\
&(\tau-1),\quad  \tau\in (1,2],
\end{cases}
\end{equation*}
is the unique solution of the problem \eqref{ex1}-\eqref{ex3}.

As discussed in the proof of the Theorem \ref{Thm1}, we have 
\begin {align*}
 L&= \max\left\lbrace \frac{L_{h_{i}}}{\Gamma(\beta)} \frac{(\sigma_{i}-\tau_{i})^{\beta-\frac{1}{2}}}{\sqrt{(2\beta-1)}\sqrt{2\theta}}+\frac{M_{f}}{\Gamma(\alpha)}\frac{(\tau_{i+1}-\sigma_{i})^{\alpha-\frac{1}{2}}}{\sqrt{(2\alpha-1)}\sqrt{2\theta}}  
+\frac{N_{f}L_{h_{i}}}{{\Gamma(\alpha+1)}} \frac{(\tau_{i+1}-\sigma_{i})^{\alpha+\frac{1}{2}}}{\sqrt{2\alpha+1}\sqrt{2\theta}} ; i=0,1\right\rbrace \\
&=\max\left\lbrace \frac{M_{f}}{\Gamma(\alpha)}\frac{1}{\sqrt{(2\alpha-1)}\sqrt{2\theta}}+\frac{N_{f}K_{h}}{{\Gamma(\alpha+1)}} \frac{1}{\sqrt{2\alpha}\sqrt{2\theta}},\right.\\ &\qquad\qquad\left.
\qquad \frac{L_{h_{1}}}{\Gamma(\beta)} \frac{1}{\sqrt{(2\beta-1)}\sqrt{2\theta}}+\frac{M_{f}}{\Gamma(\alpha)}\frac{1}{\sqrt{(2\alpha-1)}\sqrt{2\theta}}  
+\frac{N_{f}L_{h_{1}}}{{\Gamma(\alpha+1)}} \frac{1}{\sqrt{2\alpha+1}\sqrt{2\theta}} ; \right\rbrace\\
&=\frac{7.5188\sqrt{3}}{\Gamma(\frac{2}{3})\sqrt{2\theta}} 
 \end{align*}
Choose  $\theta>\left(\dfrac{7.5188\sqrt{3}}{\sqrt{2}\Gamma(\frac{2}{3})}\right)^{2}=46.2473$ so that $L<1$. For this choice of $\theta$ , we have:
for $\tau\in (0,1],$
$$ |y(\tau)-x(\tau)|e^{-\theta \tau}=|\tau-\tau^{2}|e^{-\theta \tau}\leq (\tau+\tau^{2})\leq 2 ,$$
 for $\tau\in (1,2],$
$$ |y(\tau)-x(\tau)|e^{-\theta \tau}=0,$$
and for $\tau\in (2,3],$
$$ |y(\tau)-x(\tau)|e^{-\theta \tau}=|\tau-\tau^{2}|e^{-\theta \tau}\leq (\tau+\tau^{2})e^{-2\theta }=11e^{-2\theta }.$$
Thus,$$ |y(\tau)-x(\tau)|e^{-\theta \tau}\leq C_{f,\theta,h,\alpha,\beta,h_{i},\varphi}(\psi+\varphi(\tau)) ,\;\; \tau\in J=[0,3],$$
where $ C_{f,\theta,h,\alpha,\beta,h_{i},~\varphi}=1, \psi=0~ \mbox{and}~ \varphi(\tau)=2.8361\mathbb{E}_{\frac{2}{3}}(\tau^{\frac{2}{3}})$.

\end{document}